\documentclass[12pt,reqno]{amsart}
\usepackage{mathtools}

\mathtoolsset{showonlyrefs}
\usepackage[hmarginratio={1:1},scale=0.8]{geometry}
\usepackage{enumerate}
\usepackage{amssymb}
\usepackage{cite,color}
\usepackage{bm} \numberwithin{equation}{section}

\usepackage{amsfonts} 
\DeclareMathSymbol{\shortminus}{\mathbin}{AMSa}{"39}
\newcommand{\mycomment}[1]{}

\DeclareMathOperator{\Span}{span}

\DeclareMathOperator{\Lc}{LC}
\DeclareMathOperator{\cvx}{CVX}

\DeclareMathOperator{\loc}{loc}
\DeclareMathOperator{\Div}{div}
\DeclareMathOperator{\Per}{Per}
\newcommand{\ko}{\mathcal K_o^n}
\newcommand{\lh}{\left}
\newcommand{\rh}{\right}

\newcommand{\rn}{{\mathbb R^n}}

\newcommand{\sn}{ {S^{n-1}}}

\newcommand{\ac}{\widetilde{C}^{e}_{q}}
\renewcommand{\L}{{\mathcal L}}

\newtheorem{lemma}{Lemma}[section]
\newtheorem{theorem}[lemma]{Theorem}

\newtheorem{definition}[lemma]{Definition}
\newtheorem{coro}[lemma]{Corollary}

\newtheorem{prop}[lemma]{Proposition}
\newtheorem{remark}[lemma]{Remark}
\title{Dual curvature measures for log-concave functions}

\author[Y. Huang]{Yong Huang}
\address{Institute of Mathematics, Hunan University, 2 Lushan S Road, 410082, Changsha, China}
\email{huangyong@hnu.edu.cn}
\author[J. Liu]{Jiaqian Liu}
\address{School of Mathematics and Statistics, Henan University, Kaifeng, 475001, China; Institute of Mathematics, Hunan University, 2 Lushan S Road, 410082, Changsha, China}
\email{liujiaqian@hnu.edu.cn}
\author[D. Xi]{Dongmeng Xi}
\address{Department of Mathematics, Shanghai University, 266 Jufengyuan Rd, 200444, Shanghai, China}
\email{dongmeng.xi@live.com}
\author[Y. Zhao]{Yiming Zhao}
\address{Department of Mathematics,  Syracuse University, 130 Sims Dr, 215 Carnegie, Syracuse, NY, 13244, USA}
\email{yzhao197@syr.edu}
\date{\today}
\subjclass[]{}
\thanks{Research of Huang is supported by NSFC 12171144 and 12231006. Research of Xi is supported by National Natural Science Foundation of China (12071277), and STSCM program (20JC1412600). Research of Zhao is supported, in part, by U.S. National Science Foundation Grant DMS-2132330.}

\begin{document}
\begin{abstract}
We introduce dual curvature measures for log-concave functions, which in the case of characteristic functions recover the dual curvature measures for convex bodies introduced by  Huang-Lutwak-Yang-Zhang in 2016. Variational formulas are shown. The associated Minkowski problem for these dual curvature measures is considered and sufficient conditions in the symmetric setting are demonstrated.
\end{abstract}

\maketitle

\section{Introduction}
Geometric measures associated with convex bodies have been a core part of convex geometric analysis in the past few decades. In the classical Brunn-Minkowski theory of convex bodies, \emph{quermassintegrals} (such as volume, surface area, mean width, and much more in higher dimensions) are the central geometric invariants and are used to describe the shape of convex bodies via isoperimetric (or reverse isoperimetric) inequalities. \emph{Area measures} introduced by Aleksandrov, Fenchel-Jessen, and \emph{curvature measures} introduced by Federer can be viewed as the ``derivative'' of quermassintegrals when viewed as functionals on the set of convex bodies. Invariably, these geometric measures carry some curvature terms which make it possible for them to encode shape information of convex bodies. At the same time, unlike curvatures (in the sense of differential geometry), these geometric measures are defined even with minimal regularity assumptions. The study of these geometric measures is often intertwined with PDE (Monge-Amp\`{e}re equations in particular), Gauss curvature flows, and inevitably isoperimetric inequalities. (After all, half of calculus is focused on using derivatives to study properties of the original function.)

A major alternative to the classical Brunn-Minkowski theory in modern convex geometry is the dual Brunn-Minkowski theory. The dual Brunn-Minkowski theory, introduced by Lutwak in 1975, is a theory that is in a sense dual to the classical Brunn-Minkowski theory. A good discussion of the dual theory can be found in Section 9.3 of Schneider's classical volume \cite{MR3155183}. Quoting from Gardner-Hug-Weil \cite{MR3120744}:``The dual Brunn-Minkowski theory can count among its successes the solution of the Busemann-Petty problem in \cite{MR1298719}, \cite{MR1689343}, \cite{MR963487}, and \cite{MR1689339}. It also has connections and applications to integral geometry, Minkowski geometry, the local theory of Banach spaces, and stereology; see \cite{MR2251886} and the references given there.''

In the seminal work \cite{MR3573332}, Huang-Lutwak-Yang-Zhang (Huang-LYZ), for the first time, revealed the fundamental geometric measures, duals of Federer's curvature measures, called \emph{dual curvature measures}, in the dual Brunn-Minkowski theory. These measures were obtained through ``differentiating'' \emph{dual quermassintegrals} which are central in the dual theory. They have led to many natural open problems and quickly attracted much attention. Details on this will be provided below.

It is well-known that the set of convex bodies can be embedded in the set of upper semi-continuous log-concave functions via their characteristic functions. This work aims to introduce the functional version of dual curvature measures through the machinery of the theory of functions of bounded anisotropic weighted variation and to study their characterization problem (generally known as \emph{Minkowski problems}). It is worth pointing out that, by picking proper weight functions, functional versions of many other geometric measures can be introduced in this fashion. However, this will not be explored in this work.

In the past few decades (even more so in the last decade), interest in log-concave functions has grown considerably, much of it motivated by their counterparts in the theory of convex bodies. Perhaps the first such breakthrough and by now a well-known result is the Pr\'{e}kopa-Leindler inequality: For any nonnegative integrable functions $f, g$ on $\rn$ and their sup-convolution $(1-\lambda)\cdot f\oplus \lambda\cdot g$ given by
\begin{equation}
	  \lh( (1-\lambda)\cdot f\oplus \lambda\cdot g \rh)  (z)= \sup_{(1-\lambda)x+\lambda y=z} f(x)^{1-\lambda}g(y)^\lambda,
\end{equation}
where $0<\lambda<1$, one has the following inequality regarding their $\L^1$ norms,
\begin{equation}
\label{eq 9.1.1}
	\int_\rn \lh( (1-\lambda)\cdot f\oplus \lambda\cdot g \rh) (z)dz\geq \left(\int_{\rn} f(x)dx\right)^{1-\lambda}\left(\int_\rn g(y)dy\right)^\lambda.
\end{equation}
The Pr\'{e}kopa-Leinder inequality is the functional (and equivalent) version of the celebrated Brunn-Minkowski inequality,
\begin{equation}
	V((1-\lambda)X+\lambda Y)\geq V(X)^{1-\lambda}V(Y)^\lambda,
\end{equation}
which holds for any bounded measurable sets $X, Y\subset \rn$ such that $(1-\lambda)X+\lambda Y$ is measurable.
 See the survey \cite{MR1898210} by Gardner. It is important to note that convexity is required neither in the Brunn-Minkowski inequality nor in the Pr\'{e}kopa-Leindler inequality, although it does manifest itself in the equality conditions. Notice here that the Minkowski combination $(1-\lambda)X+\lambda Y$ corresponds to the sup-convolution between functions (see, \eqref{eq 9.1.2} for a complete definition) and the volume of a measurable set corresponds to the integral of a non-negative function. That this is natural can be seen by replacing $f$ and $g$ by characteristic functions of subsets of $\rn$.

 In the same spirit, many other geometric invariants and operations have found their counterparts for log-concave functions (or, equivalently, convex functions). We provide a quick overview of some of these remarkable results. In the seminal work \cite{MR2480615}, Artstein-Avidan and Milman demonstrated how the Legendre-Fenchel transform can be viewed as the functional version of taking the polar body of a convex body. Compare this to another remarkable paper \cite{MR2438994} by B\"{o}r\"{o}czky-Schneider. Prior to this, the functional version of the Blaschke-Santal\'o inequality was discovered by Ball in his Ph.D. thesis and by Artstein-Avidan, Klartag, and Milman in \cite{MR2220210}. Steiner formula and quermassintegrals for quasi-concave functions were studied by Bobkov-Colesanti-Fragal\`{a} \cite{MR3147446}. Extensions of affine surface area and affine isoperimetric inequalities can be found in \cite{MR2899992,MR3335279,MR3551664,MR3695895}. Much more recently, Colesanti, Ludwig, and Mussnig embarked on a journey to characterize valuations on the set of convex functions \cite{MR4453228,https://doi.org/10.48550/arxiv.2109.09434,https://doi.org/10.48550/arxiv.2009.03702,MR4124129,https://doi.org/10.48550/arxiv.2201.11565} (compare them to Hadwiger-type theorems on convex bodies \cite{MR2680490,MR2772547,MR1991649,MR2966660,MR3194492,MR2668553}).

In the dual Brunn-Minkowski theory, the central geometric invariants are known as \emph{dual quermassintegrals}. Let $q=1, \cdots, n$, and $K$ be a convex body that contains the origin in its interior. As Lutwak \cite{MR963487} showed, up to a constant multiple, the $q$-th dual quermassintegral $\widetilde{V}_q(K)$ can be defined as the average of lower-dimensional sectional areas of $K$ with $q$-dimensional subspaces:
\begin{equation}
	\widetilde{V}_q(K)=c\int_{G(n,q)}\mathcal{H}^{q}(K\cap \xi)d\xi,
\end{equation}
where $G(n,q)$ is the Grassmannian manifold containing all $q$-dimensional subspaces of $\rn$ and the integration is with respect to the Haar measure. Dual quermassintegrals have integral representations (see \eqref{eq 91}) which warrant the immediate extension to $q\in \mathbb{R}$. Note that with the exception of the special case $q=n$, when the dual quermassintegral is simply volume, the $q$-th dual quermassintegral is generally \emph{not} invariant under translations of $K$.

A major question answered in the landmark work \cite{MR3573332} (and subsequently \cite{MR3783409}) by Huang-LYZ was the differentiability of $\widetilde{V}_q$. In particular, it was shown that if $K$ is a convex body in $\rn$ such that the origin is an interior point and $L$ is a compact convex subset of $\rn$, then
\begin{equation}
\label{eq 92}
	\lim_{t\rightarrow 0^+}\frac{\widetilde{V}_q(K+ t\cdot L)-\widetilde{V}_q(K)}{t}= \int_{\sn} h_L(v) \frac{1}{h_K(v)}d\widetilde{C}_q(K,v).
\end{equation}
Here the geometric measure $\widetilde{C}_q(K,\cdot)$ is known as the $q$-th \emph{dual curvature measure} of $K$. In fact, there is naturally an $L_p$ version of \eqref{eq 92} that leads to the $(p,q)$-dual curvature measure introduced in \cite{MR3783409} and the measure $\frac{1}{h_K(v)}d\widetilde{C}_q(K,v)$ is nothing but the $(1,q)$-dual curvature measure. The ``1'' here stands for the fact that the sum $K+t\cdot L$ is the classical Minkowski sum, or, the $l^1$ sum of support functions of $K$ and $L$.

Let $q>0$. The $(q-n)$-th moment of a density function $f$ is defined as
\begin{equation}
	\widetilde{V}_q(f) = \int_{\rn }|x|^{q-n}f(x)dx,
\end{equation}
if it exists. The moment $\widetilde{V}_q$ is a natural extension of dual quermassintegrals to the set of log-concave functions (which in turn justifies this notation). Indeed, if $f=1_K$, where $1_K$ is the characteristic function of  some convex body $K$ that contains the origin in its interior, then, by integration via polar coordinates, one immediately has
\begin{equation}
	\widetilde{V}_q(1_K) = \widetilde{V}_q(K).
\end{equation}
Motivated by the work \cite{MR3573332} and the correspondence between the Minkowski combination and sup-convolution, it is natural to ask whether
\begin{equation}
\label{eq 93}
	\lim_{t\rightarrow 0} \frac{\widetilde{V}_q(f\oplus t\cdot g)-\widetilde{V}_q(f)}{t}
\end{equation}
exists for log-concave functions $f,g$, and if it does, what the limit is. We remark that with the exception of $q=n$, for generic $q>0$, the moments considered in \eqref{eq 93} are \emph{not} invariant under the transformation $f(x)\rightarrow f(x+x_0)$. Therefore, the relative position of the origin is crucial in the study of \eqref{eq 93}. In fact, since \eqref{eq 92} only holds when $K$ contains the origin in its interior, some condition on $f$ that mimics this constraint is expected.

When $q=n$, the functional $\widetilde{V}_q$ is nothing but the $\L^1$ norm of a log-concave function. In this case, the limit \eqref{eq 93} was studied by Colesanti-Fragal\`{a} \cite{MR3077887} under various regularity assumptions on $f$ and $g$. In particular, they discovered that the limit \eqref{eq 93} consists of two parts---one concerning the behavior of $f$ inside its support, the other concerning the values of $f$ on the boundary of its support as well as the shape of the support set. Around the same time, Cordero-Erausquin and Klartag \cite{MR3341966} studied the limit with the additional assumption that $f$ is \emph{essentially continuous} and explored the connection with complex analysis \cite{MR3137248} and optimal transport.
Recently, Rotem \cite{https://doi.org/10.48550/arxiv.2206.13146} showed that the result of Colesanti-Fragal\`{a} remains valid without any of the various additional regularity requirements, by employing tools from the theory of functions of bounded anisotropic variation. The first main result of this paper is to show that by considering functions of bounded anisotropic \emph{weighted} variation, one can compute the limit in \eqref{eq 93} for any $q>0$. It is important to emphasize that unlike the special case $q=n$, for generic $q>0$, the $q$-th moment of a function is not translation-invariant. In particular, our approach is motivated by the dual approach (differentiating radial functions) to the variational formula developed in \cite{MR3573332}.

It is also worth pointing out that the extension of functions of bounded variation in $\rn$ with respect to the Lebesgue measure to functions of bounded variation in $\rn$ with respect to an arbitrary measure (even those absolutely continuous with respect to the Lebesgue measure) is not entirely trivial. This has been previously done in, for example, \cite{MR1864805} and \cite{MR2005202} via different approaches (which led to non-equivalent definitions of weighted total variation).

Let $\Lc(\rn)$ be the set of all upper semi-continuous log-concave functions $f:\rn \rightarrow [0,\infty)$. The limit \eqref{eq 93} leads to two Borel measures---one on $\rn$ and one on $\sn$.

 \begin{definition}
 \label{def intro 1}
	Let $f=e^{-\phi}\in \Lc(\rn)$ with nonzero finite $\L^1$ norm. The Euclidean $q$-th dual curvature measure of $f$,  denoted by $\widetilde{C}^{e}_{q}(f;\cdot)$, is a Borel measure on $\rn$ given by
	\begin{equation}
	\label{eq 8.8.1}
		\widetilde{C}^{e}_{q}(f;\mathcal{B}) = \int_{\nabla\phi(x)\in \mathcal{B}}|x|^{q-n}f(x)\,dx,
	\end{equation}
	for each Borel set $\mathcal{B}\subset \rn$.\end{definition}

In \eqref{eq 8.8.1}, note that since $\phi$ is convex, its gradient $\nabla \phi$ exists almost everywhere in the interior of  its domain $\{x\in \rn: \phi(x)<\infty\}$. Note that by definition, $f>0$ if and only if $\phi<\infty$. Therefore, the integral in \eqref{eq 8.8.1} is well-defined.

\begin{definition}
\label{def intro 2}
	Let $f=e^{-\phi}\in \Lc(\rn)$ with nonzero finite $\L^1$ norm. The spherical $q$-th dual curvature measure of $f$, denoted by $\widetilde{C}^{s}_{q}(f;\cdot)$, is a Borel measure on $\sn$ given by
	\begin{equation}
		\widetilde{C}^{s}_{q}(f;\eta) = \int_{\nu_{K_f}(x)\in \eta}|x|^{q-n}f(x)\,d\mathcal{H}^{n-1}(x),
	\end{equation}
	for each Borel set $\eta\subset \sn$, where $K_f$ is the support of $f$ and $\nu_{K_f}$ is its Gauss map defined almost everywhere on $\partial K_f$ with respect to $d\mathcal{H}^{n-1}(x)$.
\end{definition}

These two measures generated through differentiating the $q$-th moment of a log-concave function $f$ with respect to sup-convolution are associated with the absolutely continuous and singular part of the distributional derivative of $f$, respectively. It is worth noting that in the case of the characteristic function of a convex body, the measure $\widetilde{C}^{s}_{q}(f;\cdot)$ recovers the $(1,q)$-dual curvature measure for convex bodies appearing in \eqref{eq 92}.

The first of our main theorems shows that the limit in \eqref{eq 93} does exist under minor assumptions on $f$ and $g$ near the origin.
\begin{theorem}
\label{theorem intro 1}
 		Let $f=e^{-\phi}\in \Lc(\rn)$ with non-zero finite $\L^1$ norm and $q>0$. Assume $f$ achieves its maximum at the origin and
 		\begin{equation}
 		\label{eq 400a}
 			\limsup_{x\rightarrow o} \frac{|f(x)-f(o)|}{|x|^{\alpha+1}}<\infty,
 		\end{equation}
 		for some $0<\alpha<1$.
 		
 		 Let $g=e^{-\psi}\in \Lc(\rn)$ be compactly supported with $g(o)>0$. Then,
 		\begin{equation}
 		\label{eq 95}
	\lim_{t\rightarrow 0^+} \frac{\widetilde{V}_q(f\oplus t\cdot g)-\widetilde{V}_q(f)}{t} = \int_{\rn}\psi^*(y)d\ac(f;y)+\int_{\sn} h_{K_g}(v)d\widetilde{C}^{s}_{q}(f;v).
 		\end{equation}
 		where $h_{K_g}$ is the support function of the support set $K_g$ of $g$, and $\psi^*$ is the Legendre-Fenchel conjugate of $\psi$.
 \end{theorem}

 Note that hypothesis \eqref{eq 400a} in the above theorem is not the best hypothesis, see Proposition \ref{prop 8221} and Remark \ref{remark 8101} for more details.  We emphasize again that it is expected that we need some condition on $f$ that mimics the idea that ``$f$ contains the origin in its interior''. The assumption that $f$ achieves its maximum at the origin, together with hypothesis \eqref{eq 400a}, ensures that almost all of $f$'s nonempty level sets contain the origin in the interior and these level sets contain the origin in their interiors in some \emph{uniform} way. We remark that if $f$ is $C^{1,\alpha}$ in a neighborhood of the origin, then \eqref{eq 400a} is satisfied. To better explain the condition $g(o)>0$, we focus for the moment on the special case that $g$ is the characteristic function of some convex body $L$. In this case, without the condition $g(o)>0$, the convex body $L$ might be far away from the origin. As a consequence, the origin might be outside the Minkowski combination of the level sets of $f$ and $L$. This makes it very challenging to apply geometric results to level sets of $f$. See Theorem \ref{theorem delta layer cake} for details.

 When $f$ and $g$ are characteristic functions of convex bodies that contain the origin in their interiors, the first integral on the right-hand side in  \eqref{eq 95} vanishes whereas the second term becomes the right-hand side of \eqref{eq 92}.

\emph{Minkowski problems} in convex geometric analysis are characterization problems of geometric measures associated with convex bodies. These geometric measures are often ``derivatives'' of important geometric invariants. In differential geometry, Minkowski problems are known as various prescribed curvature problems. This line of research that asks when a given measure $\mu$ can be realized as a certain geometric measure of a to-be-solved convex body (without any unnecessary regularity assumptions) goes back to the classical Minkowski problem that inspired the study of nonlinear elliptic PDE through the last century; see, for example, Minkowski \cite{MR1511220}, Aleksandrov \cite{MR0001597}, Cheng-Yau \cite{MR0423267}, Pogorelov\cite{MR0478079}, and the works of Caffarelli on the regularity theory of Monge-Amp\`{e}re equations \cite{MR1005611, MR1038359,MR1038360}, among many other influential works. In many ways, the study of Minkowski problems goes hand-in-hand with the study of sharp isoperimetric inequalities; see \cite{MR1849187}.

In the last 2-3 decades, there are two major families of Minkowski problems. One is the $L_p$ Minkowski problem that belongs to the $L_p$ Brunn-Minkowski theory whose success can be credited to the landmark work by Lutwak \cite{MR1231704, MR1378681} where the fundamental $L_p$ surface area measure was discovered. The $L_p$ Minkowski problem includes the logarithmic Minkowski problem and the centro-affine Minkowski problem and has been studied through a variety of methods; see, for example, Hug-Lutwak-Yang-Zhang (Hug-LYZ) \cite{MR2132298}, Chou-Wang \cite{MR2254308}, B\"{o}r\"{o}czky-LYZ \cite{MR3037788}, and most recently Guang-Li-Wang \cite{https://doi.org/10.48550/arxiv.2203.05099}. A vast library of works on this topic can be found by looking for those citing the above-mentioned works. It is worth pointing out that there is much unknown regarding the $L_p$ Brunn-Minkowski theory, especially for $p<1$. In fact, the log Brunn-Minkowski conjecture, arguably the most beautiful and powerful (yet plausible) conjecture in convex geometry in the last decade, is the isoperimetric inequality associated with the log Minkowski problem. See, for example, \cite{MR2964630, MR4438690, MR4088419, milman2021centroaffine, vanhandel2022local}.

The other major family of Minkowski problems are the \emph{dual Minkowski problems} following the landmark work \cite{MR3573332}. In a short period since \cite{MR3573332}, there have been many influential works on this topic which have already led to many interesting conjectures regarding isoperimetric inequalities as well as the discovery of novel curvature flows; see, for example, B\"{o}r\"{o}czky-Henk-Pollehn \cite{MR3825606}, Chen-Chen-Li \cite{MR4259871}, Chen-Huang-Zhao \cite{MR3953117}, Chen-Li \cite{MR3818073}, Gardner-Hug-Weil-Xing-Ye \cite{MR3882970}, Henk-Pollehn \cite{MR3725875}, Li-Sheng-Wang \cite{MR4055992}, Liu-Lu \cite{MR4127893}, Zhao \cite{MR3880233}. It is important to note that the list is by no means exhaustive.

In this paper, we study the Minkowski problem for $\ac$.

\textbf{The functional dual Minkowski problem.} Let $q>0$ and $\mu$ be a non-zero finite Borel measure on $\rn$. Find the necessary and sufficient conditions on $\mu$ so that there exists $f\in \Lc(\rn)$ with nonzero finite $\L^1$ norm such that
\begin{equation}
\label{eq 96}
	\mu = \ac(f;\cdot).
\end{equation}

Under sufficient regularity assumption, that is, the measure $\mu$ has a $C^\infty$ density (say, $g\geq 0$) and $f\in C^{\infty}$, equation \eqref{eq 96} is equivalent to the following Monge-Amp\`{e}re type equation in $\rn$
\begin{equation}
\label{eq 99}
	g(\nabla \phi(x))\det (\nabla^2 \phi(x)) = |x|^{q-n}e^{-\phi(x)},
\end{equation}
where $f=e^{-\phi}$.

It is important to note that the measure $\ac(f;\cdot)$ might not be absolutely continuous. Thus, the Minkowski problem \eqref{eq 96} does not always reduce to \eqref{eq 99} in the general setting.

 When $q=n$, the functional dual Minkowski problem becomes the Minkowski problem for moment measures. See Cordero-Erausquin and Klartag \cite{MR3341966} where it is completely solved within the class of essentially continuous functions. The highly nontrivial $L_p$ extension of Cordero-Erausquin and Klartag's result can be found in the recent papers by Fang-Xing-Ye \cite{MR4375791} for $p>1$ and Rotem \cite{https://doi.org/10.48550/arxiv.2006.16933} for $0<p<1$.

As pointed out earlier, a key difference between the case $q=n$ and $q\neq n$ is that in the latter case, \eqref{eq 96} is not invariant under translations of $f$ (with respect to its domain). We point out that translation-invariance played a central role in \cite{MR3341966}.

In the current work, we provide a sufficient condition for the existence of solutions to  \eqref{eq 96} in the origin-symmetric case.
\begin{theorem}
\label{theorem intro 2}
	Let $q>0$ and $\mu$ be a non-zero even finite Borel measure on $\rn$. Suppose $\mu$ is not concentrated in any proper subspaces and $\int_\rn |x|d\mu(x)<\infty$. There exists an even $f_0\in \Lc(\rn)$ with nonzero finite $\L^1$ norm such that
	\begin{equation}
		\mu = \ac(f_0;\cdot).
	\end{equation}
\end{theorem}

The functional dual Minkowski problem \eqref{eq 96} is heavily intertwined with its counterpart in the setting of convex bodies. In particular, estimates regarding dual quermassintegrals are critically needed. As part of the process to get the required estimates, we require a Blaschke-Santal\'{o} type inequality for the functional $\widetilde{V}_q$. See Lemma \ref{lemma 4}. It is of great interest to see if there is a sharp (more refined) version.

We remark  that the assumption
\begin{equation}
	\int_\rn |x|d\mu(x)<\infty
\end{equation}
is necessary here. See Theorem \ref{thm necessity}.

The rest of this paper is organized in the following way. In Section 2, we recall some notations and basics. In Section 3, we gather some basics in the theory of functions of bounded variation. Section 4 is devoted to proving Theorem \ref{theorem intro 1} whereas Section 5 is devoted to proving Theorem \ref{theorem intro 2}.

\textbf{Acknowledgement.} We are in great debt to the referees for their \emph{extremely} valuable comments and suggestions.

\section{Preliminaries}

This section is divided into two parts. The first part contains some notations and basics in the theory of convex bodies, whereas the second part deals with those for convex functions as well as log-concave functions.

For convenience, throughout the current work, if the exact value of a constant $C>0$ does not matter, then we may use the same $C$ for different positive constants (that may differ from line to line).

\subsection{Convex bodies}
\label{section preliminary convex body}
The standard reference is the comprehensive book \cite{MR3155183} by Schneider.

A convex body in $\rn$ is a compact convex set with a nonempty interior. The boundary of $K$ is written as $\partial K$. We use $\mathcal{K}^n$ for the set of all convex bodies in $\rn$. The subclass of convex bodies that contain the origin in their interiors in $\rn$ is denoted by $\ko$.

We will use $B(x,r)$ to denote the ball in $\rn$ centered at $x$ with radius $r$. Occasionally, we write $B(r)=B(o,r)$ and $B=B(o,1)$ for simplicity.

The support function $h_K$ of $K$ is defined by
\begin{equation}
\label{eq local 0004}
	h_K(y) = \max\{\langle x, y\rangle  : x\in K\}, \quad y\in\rn.
\end{equation}
The support function $h_K$ is a continuous function homogeneous of degree 1. Suppose $K$ contains the origin in its interior. The radial function $\rho_K$ is defined by
\[ \rho_K(x) = \max\{\lambda : \lambda x \in K\}, \quad x\in \rn\setminus \{0\}. \]
The radial function $\rho_K$ is a continuous function homogeneous of degree $-1$. It is not hard to see that $\rho_K(u)u \in
\partial K$ for all $u\in S^{n-1}$ and the reciprocal radial function is a (potentially asymmetric) norm. To be more specific, we write
\begin{equation}
\label{eq 210}
	\|x\|_K = \frac{1}{\rho_K(x)} = h_{K^*}(x), \quad \text{for each } x\in \rn.
\end{equation}
Here, the convex body $K^*$ is known as the polar body of $K$ and is defined by
\begin{equation}
	K^*= \{y\in \rn:\langle x, y\rangle \leq 1, \forall x\in K\}.
\end{equation}
By the definition of the polar body, it is simple to see that the Banach spaces $(\rn, \|\cdot\|_{K})$ and $(\rn, \|\cdot\|_{K^*})$ are dual to each other and we have the following generalized Cauchy-Schwarz inequality
\begin{equation}
	\langle x, y\rangle \leq \|x\|_K\|y\|_{K^*}.
\end{equation}

 Let $h:\sn\rightarrow (0,\infty)$ be continuous, the \emph{Wulff shape} $[h]\in \mathcal{K}_o^n$ is given by 
\begin{equation}
	[h]=\{x\in \rn: \langle x, v\rangle\leq h(v) \text{ for all } v\in \sn\}.
\end{equation}
It is simple to see that if $K\in \mathcal{K}_o^n$, then $[h_K]= K$. Also immediate from the definition of $[h]$ is that for every $u\in \sn$, we have 
\begin{equation}
	\label{eq 884}
	\rho_{[h]}(u)\langle u, v\rangle\leq h(v), \qquad\forall v\in \sn,
\end{equation}
and there exists $v_*\in \sn$ such that
\begin{equation}
	\label{eq 885}
	\rho_{[h]}(u)\langle u, v_*\rangle= h(v_*).
\end{equation}

For each $x\in \partial K$, we will write $\nu_K(x)$ for the outer unit normal of $K$ at $x$. Note that by convexity, the map $\nu_K$ is defined almost everywhere on $\partial K$ with respect to $\mathcal{H}^{n-1}$. For each $v\in \sn$, define 
\begin{equation}
	\nu_K^{-1}(v)=\{x\in \partial K: \langle x, v\rangle=h_K(v)\}. 
\end{equation}
Since $K$ is a convex body, for almost all $v\in \sn$, the set $\nu_K^{-1}(v)$ contains only one boundary point of $K$. With slight abuse of notation, we will use $\nu_K^{-1}$ to denote a map that is defined almost everywhere on $\sn$ and $\nu_K^{-1}$ maps $v$ to the unique point in $\nu_K^{-1}(v)$.

The fundamental geometric functionals in the dual Brunn-Minkowski theory are dual quermassintegrals. For $q\neq 0$, the $q$-th dual quermassintegral of $K$, denoted by $\widetilde{V}_q(K)$, is defined as
\begin{equation}
\label{eq 91}
	\widetilde{V}_q(K) = \frac{1}{q}\int_\sn \rho_K^q(u)du.
\end{equation}
When $q=1,\cdots, n$, dual quermassintegrals have the strongest geometric significance. They are proportional to the mean of the $q$-dimensional volume of intersections of $K$ with $q$-dimensional subspaces in $\rn$.

In \cite{MR3573332}, it was established that variation of the dual quermassintegral with respect to the logarithmic Minkowski sum leads to the so-called dual curvature measure:
\begin{equation}
	\widetilde{C}_q(K,\eta) = \int_{\nu_K(x)\in \eta} \langle x, \nu_K(x)\rangle |x|^{q-n}d\mathcal{H}^{n-1}(x), \text{ for each Borel } \eta\subset \sn.
\end{equation}
In particular, this implies that for each $p\in \mathbb{R}$, we have
\begin{equation}
\label{eq 71}
	\lim_{t\rightarrow 0}\frac{\widetilde{V}_q(K+_pt\cdot L)-\widetilde{V}_q(K)}{t} = \int_\sn h_L(v)h_K^{-p}(v)d\widetilde{C}_{q}(K,v):=\int_{\sn}h_L(v)d\widetilde{C}_{p,q}(K,v),
\end{equation}
where the Borel measure $\widetilde{C}_{p,q}(K,\cdot)$ is known as the $(p,q)$-dual curvature measure of $K$. Here $K+_p t\cdot L$ is known as the $L_p$ Minkowski combination between convex bodies. In particular, when $p\geq 1$ and $t>0$, the convex body $K+_pt\cdot L$ is defined so that its support function is given by $(h_K^p + th_L^p)^{1/p}$. The variational formula \eqref{eq 71}, as well as the definition of the $L_p$ combination, can be found in \cite{MR3783409}.

\subsection{Convex functions and log-concave functions}
Let $\mu$ be a Borel measure on some set $\Omega$. We will use $\L^1(\mu, \Omega)$ for the set of all $\mu$-measurable functions $f$ on $\Omega$ with $\int_{\Omega}|f|d\mu<\infty$. The set $\L^1_{\text{loc}}(\mu, \Omega)$ consists of functions $f$ such that $f\in \L^1(\mu, \mathfrak{K})$ for every compact set $\mathfrak K\subset \Omega$. Occasionally, when $\Omega=\rn$, we may write $\L^1(\mu)$. When $\mu$ is the standard Lebesgue measure, we may simply write $\L^1(\Omega)$, or, $\L^1=\L^1(\rn)$. When $\mu$ is a finite measure, we write $|\mu|$ for its total mass.

Let $\cvx(\rn)$ be the set of all lower semi-continuous, convex functions $\phi:\rn \rightarrow (-\infty, \infty]$ and  $\Lc(\rn)$ be the set of all upper semi-continuous log-concave functions $f$ that take the form $f=e^{-\phi}$ for some $\phi\in \cvx(\rn)$.

For any function $\phi:\rn \rightarrow [-\infty, \infty]$, the Legendre-Fenchel conjugate of $\phi$, denoted by $\phi^*$, is defined as
\begin{equation}
	\phi^*(y) = \sup_{x\in \rn} \{ \langle x,  y\rangle - \phi(x)\}.
\end{equation}
Note that from the definition, it is simple to see that $\phi^*\in \cvx(\rn)$, as long as $\phi\not\equiv +\infty$. It is also straightforward from the definition that the Legendre-Fenchel transform reverses order; in other words, if $\phi_1\leq \phi_2$, then $\phi_1^*\geq \phi_2^*$.

When restricting to $\cvx(\rn)$, the Fenchel-Moreau theorem states that the Legendre-Fenchel transform is an involution:
\begin{equation}
	\phi^{**}=\phi, \text{ for each }\phi\in \cvx(\rn).
\end{equation}

In a remarkable paper \cite{MR2480615}, Artstein-Avidan and Milman showed that any order-reversing involution on $\cvx(\rn)$ is essentially the Legendre-Fenchel transform.

In general (without assuming $\phi\in \cvx(\rn)$), by the definition of the Legendre-Fenchel transform, one may show that
\begin{equation}
	\phi^{**}\leq \phi,
\end{equation}
and if $\phi\geq 0$, then $\phi^{**}\geq 0$.

The expression $1_E$ denotes the characteristic function of some subset $E\subset \rn$; that is, $1_E(x)=1$ if $x\in E$ and $1_E(x)=0$ if $x\notin E$.

Let $K \in \mathcal{K}^n$ and
\begin{equation}
	\psi(x)= \begin{cases}
		0, & \text{ if }x\in K\\
		\infty, & \text{otherwise}.
	\end{cases}
\end{equation}
Note that $1_K = e^{-\psi}$. It follows from the definition that 
\begin{equation}
\label{eq 8820}
	\psi^* = h_K.
\end{equation}
Indeed, let $y\in \rn$ be arbitrary. Then, according to the definition of $\psi^*$, we have
\begin{equation}
	\psi^*(y) = \sup_{x\in \rn}\{\langle x,y\rangle-\psi(x)\}= \sup_{x\in K} \{\langle x, y\rangle\}=h_K(y),
\end{equation}
where the second equality follows from the fact that if $x\notin K$, then $\psi(x)=\infty$ and consequently one can restrict to $K$ in search of the supremum.

We shall require the following trivial facts.

\begin{prop}
\label{prop 12}
	Let $f=e^{-\phi}\in \Lc(\rn)$ and $q>0$. If
\begin{equation}
\label{eq 500}
	\liminf_{|x|\rightarrow \infty} \frac{\phi(x)}{|x|}>0,
\end{equation}	
then
\begin{equation}
	\widetilde{V}_q(f)=\int_\rn f(x)|x|^{q-n}dx<\infty.
\end{equation}
\end{prop}
\begin{proof}
By \eqref{eq 500} and the convexity of $\phi$, there exist $C>0$ and $r_0>0$ such that
\begin{equation}
	\phi(x)>C|x|, \text{ for all } x\in B(r_0)^c.
\end{equation}
Therefore, we have
\begin{equation}
	\int_{B(r_0)^c} f(x)|x|^{q-n}dx\leq \int_{B(r_0)^c} e^{-C|x|}|x|^{q-n}dx<\infty.
\end{equation}
Since $f$ is upper semi-continuous, it is locally bounded from above. Therefore,
\begin{equation}
		\int_{B(r_0)} f(x)|x|^{q-n}dx\leq C_1\int_{B(r_0)}|x|^{q-n}dx<\infty.
	\end{equation}
\end{proof}

It is well-known that \eqref{eq 500} holds if and only if either of the following two statements holds:
\begin{enumerate}
	\item $f\in \L^1$;
	\item $\lim_{|x|\rightarrow \infty}\phi(x)=\infty$.
\end{enumerate}
See, for example, \cite{MR3341966}.

When a convex function $\phi$ is finite in a neighborhood of the origin, \cite[Theorem 11.8(c)]{MR1491362} combined with Proposition \ref{prop 12} immediately implies the following.

\begin{prop}
\label{prop 51}
	Let $q>0$. If $\phi:\rn\rightarrow [0,\infty]$ is finite in a neighborhood of the origin, then
	\begin{equation}
		\widetilde{V}_q(e^{-\phi^*})=\int_\rn |y|^{q-n} e^{-\phi^*(y)}dy<\infty.
	\end{equation}
\end{prop}

\section{Functions of anisotropic weighted total variation}
\label{sec BV}

Let $\Omega\subset \rn$ be an open set. Write $V\Subset \Omega$ if an open set $V$ is compactly contained in $\Omega$, that is, the closure of $V$  is compact and is a subset of $\Omega$. The set $C_c^1(V,\rn)$ consists of all $C^1$ functions from $V$ to $\rn$ with compact support. 
 We say a function $f\in \L^1_{\loc}(\Omega)$ is locally of bounded variation (\emph{i.e.}, $f\in BV_{\loc}(\Omega)$) if for each open $V\Subset \Omega$, we have
\begin{equation}
\label{eq 101}
	TV(f; V) = \sup\left\{\int_{V} f\Div T\, dx: T\in C_c^1(V, \rn), |T(x)|\leq 1, \forall x\right\}<\infty.
\end{equation}
Intuitively speaking, functions of locally bounded variation are those whose distributional derivatives are Radon measures. Indeed, the Structure Theorem for $BV_{\loc}$ functions states (see Theorem 5.1 in \cite{MR3409135}) that if $f\in BV_{\loc}(\Omega)$, then there exist a Radon measure $\|Df\|$ on $\Omega$ and a $\|Df\|$-measurable map $\sigma_f:\Omega\rightarrow \rn$ such that $|\sigma_f|=1$ $\|Df\|$-almost everywhere with
\begin{equation}
\label{eq 104}
	\int_{\Omega} f\Div T\, dx = -\int_\Omega \langle T, \sigma_f\rangle d\|Df\|,
\end{equation}
for all $T\in C_c^1(\Omega, \rn)$. When $f\in \L^1(\Omega)$ and $\|Df\|(\Omega)$ is finite, we say $f$ is of bounded variation on $\Omega$; that is $f\in BV(\Omega)$. The space of $BV(\Omega)$ is well studied and we refer the readers to the classical books \cite{MR3409135, MR1857292} for additional properties of $BV$ functions.

There have been several generalizations to the definition of $BV(\Omega)$ (and correspondingly $BV_{\loc}(\Omega)$). One direction of such generalization is that the Euclidean norm (applied to $T$) in \eqref{eq 101} can be replaced by any (potentially asymmetric) norm. Fix $L\in \mathcal{K}_o^n$, let $\|\cdot\|_{L}$ and $\|\cdot\|_{L^*}$ be as defined in \eqref{eq 210}. Note that since $L$ is compact and contains the origin in its interior, both $\|\cdot\|_L$ and $\|\cdot\|_{L^*}$ are equivalent to the standard Euclidean norm and therefore, the space $BV_L(\Omega)$ (and $BV_{L,\loc} (\Omega)$, resp.) consisting of $\L^1(\Omega)$ functions with
\begin{equation}
	TV_L(f; \Omega) = \sup\left\{\int_{\Omega} f\Div T\, dx: T\in C_c^1(\Omega, \rn), \|T(x)\|_L\leq 1, \forall x \right\}<\infty,
\end{equation}
remains unchanged when compared to $BV(\Omega)$ (and $BV_{\loc}(\Omega)$, resp.). However, the anisotropic total variation $TV_L(f; \Omega)$ is generally not the same as $TV(f;\Omega)$. As a matter of fact, when $f=1_K$ for some $K\in \mathcal{K}^n$, then $TV(1_K;\rn)$ gives the surface area of $K$ whereas $TV_L(1_K; \rn)$ gives the mixed volume $V_1(K,L)=\int_{\sn}h_L dS_K$. For $f\in BV_{\loc}(\Omega)$, we may define the anisotropic total variation measure with respect to $L$ by
\begin{equation}
	\|\shortminus Df\|_{L^*} = h_L(\shortminus \sigma_f) \|Df\|.
\end{equation}
It can be shown that $f\in BV_L(\Omega)$ if and only if $\|\shortminus Df\|_{L^*}$ is a finite measure. Moreover, we have $TV_L(f; \Omega)=\|\shortminus Df\|_{L^*}(\Omega)$. Functions of bounded anisotropic total variation were studied in, for example, \cite{MR3055988, MR2672283} where anisotropic isoperimetric inequalities and anisotropic Sobolev inequalities were studied for sets of finite perimeter and functions of bounded variation. It is important to note that many of the classical results mentioned in the standard books \cite{MR3409135, MR1857292} work in the anisotropic setting with only very minor alterations to the proofs.

Another direction of generalization to $BV(\Omega)$ is to replace the Lebesgue measure in $\rn$ by a generic measure. Things start to get complicated in this setting. As an example, the approximation of such $BV$ functions by smooth ones might fail. This explains why there are non-equivalent ways of defining $BV$ functions in a generic measure space $(\rn, \mu)$. We mention \cite{MR2005202} for one of the approaches where $\mu$ is a doubling measure. When $\mu = \omega(x)dx$ is absolutely continuous with respect to Lebesgue measure, another way of generalizing the classical total variation (not equivalent to the one given in \cite{MR2005202}; see Section 5.1 in \cite{MR2005202}) was given in \cite{MR1864805}.

For our purpose, we adopt the following definition. We say a function $f\in \mathcal{L}^1(\omega dx, \Omega)$ is of bounded anisotropic weighted variation (or $f\in BV_{L, \omega}(\Omega))$ if $f\in BV_{\loc}(\Omega)$ and $\omega\in \L^1(\|\shortminus Df\|_{L^*}, \Omega)$. In this case, we define the $(L,\omega)$-anisotropic total variation of $f$ to be
	\begin{equation}
	\label{eq 21}
		 TV_{L,\omega}(f;\Omega)=\int_{\Omega}\omega d\|\shortminus Df\|_{L^*}.
	\end{equation}
To see how this is connected to the classical definition \eqref{eq 101}, we mention that when $\omega: \rn \rightarrow [0,\infty]$ is a lower semi-continuous function with $\omega(x)>0$ for all $x\neq o$, using an approximation argument in both $f$ and $\omega$, one can see that
\begin{equation}
	TV_{L, \omega}(f;\Omega)= \sup \left\{\int_{\Omega} f\Div T\, dx: T\in C_c^1(\Omega, \rn), \|T(x)\|_L\leq \omega(x), \forall x \right\}.
\end{equation}
Since this representation is not required in the current work, we do not provide a proof here.

 Let $E\subset \rn$ be a measurable set. When $1_E\in BV_{L,\omega}(\rn)$, we say $E$ has finite $(L,\omega)$-anisotropic weighted perimeter and write $\Per_{L,\omega}(\partial E) =TV_{L,\omega}(1_E;\rn)$.

 Let $f:\rn\rightarrow \mathbb{R}$ and $t\in \mathbb{R}$, write
 \begin{equation}
 	[f>t]=\{x\in \rn: f(x)>t\}.
 \end{equation}

 For $BV$ functions, the following version of the classical coarea formula can be found in Figalli-Maggi-Pratelli \cite[(2.22)]{MR2672283}: if $f\in BV(\rn)$ and $\zeta:\rn \rightarrow [0,\infty]$ is a Borel function, then
 \begin{equation}
 	\int_\rn \zeta d\|\shortminus Df\|_{L^*}  = \int_{-\infty}^\infty\left( \int_{\rn} \zeta d\|\shortminus D1_{[f>t]}\|_{L^*}\right)dt.
 \end{equation}
 In particular, this implies
\begin{equation}
 \label{eq 26}
 	TV_{L,\omega}(f;\rn) = \int_{-\infty}^\infty \Per_{L,\omega}(\partial [f>t])dt.
 \end{equation}

 \section{Log-concave functions and their $(L,\omega)$ total variation}
 \label{section log concave variation formula}

Throughout this section, if not specified otherwise, we let $q>0$.

  It is well-known that the set of convex bodies can be embedded naturally into $\Lc(\rn)$ via their characteristic functions. Let $f = e^{-\phi}$, $g=e^{-\psi}$ be in $\Lc(\rn)$ and $s,t>0$. The sup-convolution between $f$ and $g$ can be defined via the Legendre-Fenchel conjugate of their respectively associated convex functions:
\begin{equation}
\label{eq 9.1.2}
	s\cdot f\oplus t\cdot g= e^{-(s\phi^*+t\psi^*)^*}.
\end{equation} 
When $s=1-t$, \eqref{eq 9.1.2} coincides with \eqref{eq 9.1.1}. For the purpose of the current work, we consider the special case $s=1$, which we will write as $f\oplus t\cdot g$. It is well-known that when $f=1_K$ and $g=1_L$ are characteristic functions of convex bodies, then
\begin{equation}
	1_K\oplus t\cdot 1_L = 1_{K+tL},
\end{equation}
where $K+tL$ is the usual Minkowski combination between convex bodies. Therefore, the supremum convolution $\oplus$ can be viewed as a natural generalization of the Minkowski addition for convex bodies.

For each $q>0$, the $(q-n)$-th moment of a log-concave function $f$ is defined as
\begin{equation}
	\widetilde{V}_q(f) = \int_\rn |x|^{q-n}f(x)dx.
\end{equation}
When $f=1_K$ for some convex body $K\in \ko$, by polar coordinates, it is simple to see that
\begin{equation}
	\widetilde{V}_q(1_K) = \frac{1}{q}\int_{\sn}\rho_K^q(u)du = \widetilde{V}_q(K),
\end{equation}
where $\widetilde{V}_q(K)$ is the $q$-th dual quermassintegral of $K$. Therefore, the quantity $\widetilde{V}_q$ on $\Lc(\rn)$ can be viewed as the natural generalization of dual quermassintegrals for convex bodies.

In the seminal work \cite{MR3573332, MR3783409}, the differentials of dual quermassintegrals were studied, which led to a family of long-sought-for geometric measures known as \emph{$(p,q)$-dual curvature measures}. These measures and their characterization problems (called \emph{Minkowski-type problems}) have been intensively studied in the past few years and have already led to many interesting conjectures regarding isoperimetric inequalities as well as the discovery of novel curvature flows.

It is natural to wonder whether the same philosophy can be applied in the space of log-concave functions---given that all the elements (dual quermassintegrals and Minkowski addition) have their natural counterparts in the larger space. It is the intention of the current section to demonstrate that the answer is yes, with some minor restrictions on the log-concave functions $f$ and $g$. To explain why these restrictions are needed, note that the variational formulas demonstrated in \cite{MR3573332, MR3783409}  require the convex body to contain the origin as an interior point and consequently the family of $(p,q)$-dual curvature measures are only defined for such bodies. See, for example, equation (1.9) in \cite{MR3573332}. Therefore, it is natural to expect certain restrictions on $f$ that mimic the requirement that $K$ has the origin in the interior as in the convex body case.

In the following, we show that the variation of the moment of $f$ is strongly connected to the theory of functions of bounded anisotropic weighted variation in Section \ref{sec BV}.

We will study the existence of the following limit
\begin{equation}
\label{eq 33}
	\delta_q(f,g) = \lim_{t\rightarrow 0^+} \frac{\widetilde{V}_q(f\oplus t\cdot g)-\widetilde{V}_q(f)}{t}.
\end{equation}
Note that it is not clear at all why the limit should exist.

For the rest of the section, we write $\omega_q$ for the weight function
\begin{equation}
	\omega_q(x) = |x|^{q-n},
\end{equation}
for $q>0$. It is important to emphasize that the following proofs actually work for more general weight functions. As an example, the proofs (with very minor modifications) will work for the Gaussian weight function $\omega(x)=e^{-|x|^2/2}$.

We require the following lemma from \cite{MR3573332} (see also Theorem 6.5 in \cite{MR3783409}).
\begin{lemma}[\cite{MR3573332}]\label{lemma 122} Let $K\in \mathcal{K}_o^n$ and $g:\sn \rightarrow \mathbb{R}$ be a continuous function. For sufficiently small $|t|$, define $h_t:\sn\rightarrow (0,\infty)$ by 
\begin{equation}
	h_t = h_K+tg.
\end{equation}
Then, we have
\begin{equation}
\label{eq 35}
	\lim_{t\rightarrow 0}\frac{\widetilde{V}_q([h_t])-\widetilde{V}_q([h_0])}{t} = \int_\sn g(v)d\widetilde{C}_{1,q}(K,v),
\end{equation}	
where the definition of $\widetilde{C}_{1,q}(K,\cdot)$ is given in Section \ref{section preliminary convex body}.
\end{lemma}

When $g=h_L$ for some compact convex $L\subset\rn$, we will denote the right-hand-side of \eqref{eq 35} by $\widetilde{V}_{1,q}(K,L)$; that is
\begin{equation}
\label{eq 8822}
	\widetilde{V}_{1,q}(K,L)= \int_\sn h_L(v)d\widetilde{C}_{1,q}(K,v)= \int_\sn \frac{h_L(v)}{h_K(v)}d\widetilde{C}_{q}(K,v).
\end{equation}

Recall that a log-concave function is almost everywhere differentiable.
It was shown in Rotem \cite{https://doi.org/10.48550/arxiv.2006.16933} that if $f\in \Lc(\rn)$ and $f$ is $\L^1$, then $f \in BV(\rn)$ and its distributional derivative is given by 
\begin{equation}
\label{eq 32}
	\sigma_f\|Df\| = \nabla f\,dx-f\nu_{K_f} d\mathcal{H}^{n-1}|_{\partial K_f},
\end{equation}
where $K_f$ is the support of $f$ and is thus convex by the fact that $f\in \Lc(\rn)$. With this in mind, by \eqref{eq 21}, it is simple to see that if $K,L\in \mathcal{K}_o^n$, then $K$ is of finite anisotropic weighted perimeter. Indeed,
\begin{equation}
\label{eq 31}
\begin{aligned}
	\Per_{L,\omega_q}(\partial K) &= TV_{L,\omega_q}(1_K, \rn) \\
	&= \int_{\partial K} h_L(\nu_{K}(x))|x|^{q-n}dx \\
	&= \int_{\sn} h_L(v)d\widetilde{C}_{1,q}(K,v)\\
	&=\widetilde{V}_{1,q}(K,L)<\infty.
\end{aligned}
\end{equation}

The following lemma is due to Huang-LYZ \cite{MR3573332}. We provide a short proof here for the convenience of the readers as the lemma as stated here is buried in the long proof of Lemma 4.1 in \cite{MR3573332}.
\begin{lemma}
	\label{lemma 881}
	Let $h_0,h_1: \sn \rightarrow (0,\infty)$ be continuous.
	Denote $K_i = [h_i]$. Then, for every $u\in \sn$, we have
	\begin{equation}
		\left|\log\rho_{K_1}(u)-\log \rho_{K_0}(u)\right|\leq \max_\sn |\log h_{1}-\log h_0|.
	\end{equation}
\end{lemma}
\begin{proof}
	We fix an arbitrary $u\in \sn$. By \eqref{eq 884} and \eqref{eq 885}, for each $i= 0,1$, we have 
	\begin{equation}
	\label{eq 888}
		\rho_{K_i}(u)\langle u, v\rangle \leq h_i(v), \qquad \forall v\in \sn
	\end{equation}
	and there exists $v_i\in \sn$ such that 
	\begin{equation}
	\label{eq 887}
		\rho_{K_i}(u)\langle u, v_i\rangle= h_i(v_i).
	\end{equation}
	
	By \eqref{eq 887} and \eqref{eq 888},
	\begin{equation}
	\label{eq 889}
		\begin{aligned}
			\log \rho_{K_1}(u)-\log\rho_{K_0}(u)&=\log \rho_{K_1}(u)-\log h_0(v_0)+\log \langle u, v_0\rangle\\
			&\leq \log h_1(v_0)- \log  h_0(v_0)
		\end{aligned}
	\end{equation}
	Reversing the role of $K_1$ and $K_0$ immediately gives
	\begin{equation}
		\label{eq 8810}
		\log \rho_{K_0}(u)-\log\rho_{K_1}(u)\leq \log h_0(v_1)-\log h_1(v_1).
	\end{equation}
	The desired result immediately follows from \eqref{eq 889} and \eqref{eq 8810}.
\end{proof}
\begin{lemma}
\label{lemma a11}
 	Let $K \in \mathcal{K}_o^n$ and $r_0>0$ be such that
 	\begin{equation}
 		B(r_0)\subset K.
 	\end{equation}
 	Let $L$ be a compact convex subset of $\rn$ and denote
 	\begin{equation}
 		K_t = K+tL.
 	\end{equation}
 	Then, there exist $C>0, \delta_0>0$ which depend only on $r_0$ and $\max_{L}|x|$, such that
 	\begin{equation}
 		\left|\frac{\log \rho_{K_t}-\log \rho_K}{t}\right|<C \qquad \text{on }\sn,
 	\end{equation}
 	for every $t\in (0,\delta_0)$.
 \end{lemma}

 \begin{proof}
 Since $K\in \mathcal{K}_o^n$ and $L$ is a compact convex set, for sufficiently small $t>0$ dependent only on $r_0$ and $\max_L |x|$, we have that $K_t\in \mathcal{K}_o^n$. For simplicity, we will write $h_t= h_{K_t}$. Note that $h_t$ is a positive continuous function on $\sn$ and $[h_t]=K_t$. 

 Note that on $\sn$,
 	\begin{equation}
 	\label{eq 8812}
 	\begin{aligned}
 		\log h_{t} &= \log (h_0+th_L) \\
 		&= \log h_0 + \log(1+th_L/h_K) \\
 		&= \log h_0+ t\frac{h_L}{h_K} + o(t,\cdot),
 	\end{aligned}
 	\end{equation}
 	where
 	\begin{equation}
 	\label{eq 8813}
 		|o(t,v)| \leq \frac{h_L(v)^2}{2(h_K(v)-|th_L(v)|)^2}t^2, \quad v\in \sn.
 	\end{equation}

 	It is simple to see that there exist $\delta_0, C_1>0$ that only depend on $r_0$ and $\max_{L}|x|$, such that for each $0<t<\delta_0$, we have
 	\begin{equation}
 	\label{eq 8910}
 		\left|\frac{o(t, v)}{t}\right|<C_1,
 	\end{equation}
 	uniformly in $t$ and $v$. By Lemma \ref{lemma 881}, \eqref{eq 8812}, and \eqref{eq 8910}, for each fixed $0<t<\delta_0$
 	\begin{equation}
 		\left|\frac{\log \rho_{K_t}-\log \rho_K}{t}\right|\leq  \frac{\max_{v\in \sn}\left|\log h_{t}(v)-\log h_{0}(v)\right|}{t}<C,
 	\end{equation}
 	where $C>0$ only depends on $r_0$, and $\max_{L}|x|$.
 \end{proof}

 \begin{coro}
 \label{corollary 11}
 	Under the same assumptions as in Lemma \ref{lemma a11}, for each $q>0$, there exists $\delta_0>0$ dependent only on $r_0$ and $\max_L |x|$ such that 
 	\begin{equation}
 		\left|\frac{\rho^q_{K_t}-\rho^q_K}{t}\right|<2^qCq\rho_K^q, \quad \forall t\in (0,\delta_0).
 	\end{equation}
 Here $C$ is from Lemma \ref{lemma a11}.
 \end{coro}
 \begin{proof}
 By Lemma \ref{lemma a11} and the mean value theorem,
 	\begin{equation}
 		\begin{aligned}
 			\left|\frac{\rho^q_{K_t}-\rho^q_K}{t}\right|&= \left|\frac{e^{q\log \rho_{K_t}}-e^{q\log \rho_{K}}}{\log \rho_{K_t}-\log \rho_{K}}\right|\left|\frac{\log \rho_{K_t}-\log \rho_K}{t}\right|\\
 			&\leq C\left|\frac{e^{q\log \rho_{K_t}}-e^{q\log \rho_{K}}}{\log \rho_{K_t}-\log \rho_{K}}\right|\\
 			& = Cq\theta^q,
 		\end{aligned}
 	\end{equation}
 	where $\theta$ is between $\rho_{K_t}$ and $\rho_K$. Since $B(r_0)\subset K$ and $L$ is compact, for sufficiently small $\delta_0$ (dependent on $r_0$ and $\max_L |x|)$, we have $K_t\subset 2K$ for each $0<t<\delta_0$. Consequently, $\theta<2\rho_K$, which immediately gives the desired bound.
 \end{proof}

Denote $[f\geq s]=\{x\in \rn: f(x)\geq s\}$.

 \begin{theorem}
  \label{theorem delta layer cake}
 	Let $q>0$, $f\in \Lc(\rn)$ with non-zero finite $\L^1$ norm and $L$ be a compact convex subset of $\rn$ with $o\in L$. Assume $f$ achieves its maximum at $o$, and
 		\begin{equation}
 		\label{eq 400}
 			\limsup_{x\rightarrow o} \frac{|f(x)-f(o)|}{|x|^{\alpha+1}}<\infty,
 		\end{equation}
 		for some $0<\alpha<1$.
 	Then,
 	\begin{equation}
 		\delta_q (f, 1_L)= \int_{0}^\infty \widetilde{V}_{1,q}([f\geq s],L)\,ds<\infty.
 	\end{equation}
 \end{theorem}

 \begin{proof}
 	
 	For simplicity, write
 	\begin{equation}
 		K_s = [f\geq s],
 	\end{equation}
 	and $M=f(o)=\max f$.
 	
We first claim that there exist $\varepsilon_0>0$ and $c_0>0$ such that
 	\begin{equation}
 	\label{eq 8818}
 		K_s\supset c_0(M-s)^{\frac{1}{\alpha+1}} B,
 	\end{equation}
 	for any $s\in (M-\varepsilon_0, M)$. Indeed by \eqref{eq 400}, there exist $\Lambda>0$ and $\eta_0>0$ such that for every $x\in B(\eta_0)$, we have
 	\begin{equation}
 		M-f(x)=|f(x)-f(o)|<\Lambda|x|^{\alpha+1},
 	\end{equation}
 	where we used the fact that $M=f(o)=\max f$. Equivalently, this implies that for every $x\in B(\eta_0)$, we have
 	\begin{equation}
 		f(x)>M-\Lambda|x|^{\alpha+1}.
 	\end{equation}
 	A direct computation now shows that if $s\in (0,M)$, then
 	\begin{equation}
 		K_s\supset B\left(\left(\frac{M-s}{\Lambda}\right)^{\frac{1}{\alpha+1}}\right)\cap B(\eta_0).
 	\end{equation}
 We now choose $\varepsilon_0>0$ so that for every $s\in (M-\varepsilon_0, M)$, we have 
\begin{equation}
	B\left(\left(\frac{M-s}{\Lambda}\right)^{\frac{1}{\alpha+1}}\right)\subset B(\eta_0).
\end{equation}
Consequently, we have
\begin{equation}
	K_s\supset B\left(\left(\frac{M-s}{\Lambda}\right)^{\frac{1}{\alpha+1}}\right)=c_0 (M-s)^\frac{1}{\alpha+1}B,
\end{equation}
for some $c_0>0$.
 
 In particular, \eqref{eq 8818} implies, for each $s<M$, that the set $K_s$ contains the origin in the interior. We also require that $\varepsilon_0$ is sufficiently small so that $K_s\subset c_1B$ for each $s\in (M-\varepsilon_0, M)$ and some $c_1>0$. This is possible since $f\in \L^1 \cap \Lc(\rn)$ implies that $\lim_{|x|\rightarrow \infty}f(x)=0$.
 	
 	Note that by layer cake representation, we have
 	\begin{equation}
 		\delta_q (f, 1_L) = \lim_{t\rightarrow 0^+} \int_{0}^M\frac{ \widetilde{V}_q(K_s+tL)-\widetilde{V}_q(K_s)}{t}ds
 	\end{equation}
 	
 	Step 1:
 	\begin{equation}
 		\lim_{t\rightarrow 0^+} \int_{M-\varepsilon_0}^{M}\frac{ \widetilde{V}_q(K_s+tL)-\widetilde{V}_q(K_s)}{t}ds=\int_{M-\varepsilon_0}^{M}\lim_{t\rightarrow 0^+}\frac{ \widetilde{V}_q(K_s+tL)-\widetilde{V}_q(K_s)}{t}ds.
 	\end{equation}
 	In particular, the integral on the right is finite.
 	\begin{proof}[Proof of step 1:]
 		For $t\in (0,1)$ and $s\in (M-\varepsilon_0, M)$, let
 		
 		\begin{equation}
 			g(t;s)=\widetilde{V}_q(K_s+tL).
 		\end{equation}
 		By Lemma \ref{lemma 122}, for each $s\in (M-\varepsilon_0,M)$, the function $g(t;s)$ is differentiable in $t$ and
 		\begin{equation}
 			\frac{\partial}{\partial t} g(t;s) = \int_{\sn} \frac{h_L(v)}{h_{K_s+tL}(v)}d\widetilde{C}_q(K_s+tL,v).
 		\end{equation}
 		Therefore, by the mean value theorem
 		\begin{equation}
 			\frac{g(t;s)-g(0;s)}{t} = \int_{\sn} \frac{h_L(v)}{h_{K_s+\theta L}(v)}d\widetilde{C}_q(K_s+\theta L,v),
 		\end{equation}
 		where $\theta\in [0,t]$ and is dependent on $s$. Since $o\in L$, we have $K_s+\theta L\supset K_s\supset c_0(M-s)^{\frac{1}{1+\alpha}} B$. Since $L$ is compact, for $t\in (0,1)$, we have,
 		\begin{equation}
 		\begin{aligned}
 			\frac{g(t;s)-g(0;s)}{t}&\leq q\max_{L}|x|\cdot \frac{1}{c_0}(M-s)^{-\frac{1}{1+\alpha}}\widetilde{V}_q(K_s+L)\\
 			&\leq q\max_{L}|x|\cdot\frac{1}{c_0}(M-s)^{-\frac{1}{1+\alpha}}\widetilde{V}_q(c_1B+L)\\
 			&\leq C(M-s)^{-\frac{1}{1+\alpha}},
 		\end{aligned}	
 		\end{equation}
 		for some $C>0$.
 		
 		Note that $(M-s)^{-\frac{1}{1+\alpha}}$ is integrable near $M$ thanks to $\alpha>0$. Therefore, by the dominated convergence theorem, we get the desired result.
 	\end{proof}
 	Step 2:
 	\begin{equation}
 		\lim_{t\rightarrow 0^+} \int_0^{M-\varepsilon_0}\frac{ \widetilde{V}_q(K_s+tL)-\widetilde{V}_q(K_s)}{t}ds=\int_0^{M-\varepsilon_0}\lim_{t\rightarrow 0^+}\frac{ \widetilde{V}_q(K_s+tL)-\widetilde{V}_q(K_s)}{t}ds.
 	\end{equation}
 	In particular, the integral on the right is finite.

 	\begin{proof}[Proof to step 2:]
		For each $s\in (0,M-\varepsilon_0)$, there exists $\lambda_0>0$ such that
		\begin{equation}
			K_s\supset K_{M-\varepsilon_0}\supset \lambda_0B.
		\end{equation} 	
		
	By Corollary \ref{corollary 11}, for $t\in (0,\delta_0)$ (where $\delta_0$ is from Corollary \ref{corollary 11}),
	\begin{equation}
		\begin{aligned}
			\frac{ \widetilde{V}_q(K_s+tL)-\widetilde{V}_q(K_s)}{t}= \frac{1}{q}\int_{\sn} \frac{\rho^q_{K_s+tL}-\rho_{K_s}^q}{t}du\leq 2^qC\int_{\sn} \rho_{K_s}^qdu.
		\end{aligned}
	\end{equation}
	Here the constant $C$ is from Corollary \ref{corollary 11}. In particular, $C$ and $\delta_0$ are not dependent on $s$.
	
	Note now that
	\begin{equation}
		\frac{1}{q}\int_{0}^{M-\varepsilon_0}\int_{\sn} \rho_{K_s}^qdu ds= \int_{0}^{M-\varepsilon_0}\widetilde{V}_q(K_s)ds\leq \int_{0}^M\widetilde{V}_q(K_s)ds = \int_{\rn}|x|^{q-n}f(x)dx<\infty.
	\end{equation}
	
	Therefore, we may use the dominated convergence theorem to justify the exchange of limit and integration.
 	\end{proof} 	
 	
 	By Step 1 and Step 2, we have
 	\begin{equation}
 		\begin{aligned}
 			\delta_q (f, 1_L) &=\int_{0}^M \lim_{t\rightarrow 0^+}  \frac{ \widetilde{V}_q(K_s+tL)-\widetilde{V}_q(K_s)}{t}ds\\
 			&=\int_{0}^M \widetilde{V}_{1,q}([f\geq s],L)ds\\
 			&=\int_0^\infty \widetilde{V}_{1,q}([f\geq s],L)ds.
 		\end{aligned}
 	\end{equation}
 	
 	Note that it is included in Steps 1 and 2 that the right-hand side is finite.
 \end{proof}

 We first show the validity of Theorem \ref{theorem intro 1} in the special case that $g$ is a constant multiple of a characteristic function.
 \begin{theorem}
 \label{theorem 11}
 		Let $q>0$, $f\in \Lc(\rn)$ with non-zero finite $\L^1$ norm and $L$ be a compact convex subset of $\rn$ with $o\in L$. Assume $f$ achieves its maximum at $o$ and
 		\begin{equation}
 		\label{eq 400aa}
 			\limsup_{x\rightarrow o} \frac{|f(x)-f(o)|}{|x|^{\alpha+1}}<\infty,
 		\end{equation}
 		for some $0<\alpha<1$.
 		
 		Let $g=c1_L=e^{-\psi}$ for some $c>0$. Then,
 		\begin{equation}
 			\delta_q(f,g) = \int_\rn \psi^*(y)d\ac(f;y)+ \int_{\sn} h_L(v)d\widetilde{C}_q^s(f;v).
 		\end{equation}
 \end{theorem}
 \begin{proof}
 We first restrict ourselves to the case where $L\in \mathcal{K}_o^n$.
 
 	We may assume without loss of generality that $c=1$. In this case, by \eqref{eq 8820}, $\psi^*=h_L$.
 	
 	According to Theorem \ref{theorem delta layer cake}, \eqref{eq 31}, and \eqref{eq 26},  we have
 	\begin{equation}
 		\begin{aligned}
 			\infty>\delta_q (f, 1_L)&= \int_{0}^\infty \widetilde{V}_{1,q}([f\geq s],L)\,ds\\
 			&=\int_{0}^{\infty} \Per_{L,\omega_q}(\partial [f\geq s])ds\\
 			&=TV_{L,\omega_q}(f;\rn).
 		\end{aligned}
 	\end{equation}
 	By \eqref{eq 21} and \eqref{eq 32}, we have
 	\begin{equation}
 	\label{eq 8821}
 	\begin{aligned}
 		\delta_q (f, 1_L)=TV_{L,\omega_q}(f;\rn) &= \int_\rn h_L(\nabla \phi)f(x)|x|^{q-n}dx+ \int_{\partial K_f}h_L(\nu_{K_f})f(x)|x|^{q-n}d\mathcal{H}^{n-1}(x)\\
 		&=\int_\rn h_L(y)d\ac(f;y)+\int_\sn h_L(v)d\widetilde{C}_q^s(f;v).
 	\end{aligned}
 	\end{equation}
 	Here, the last equality follows straightly from the definition of $\ac$ and $\widetilde{C}_q^s$. See Definitions \ref{def intro 1} and \ref{def intro 2}.
 	
 	To see that the result still holds when $L$ is a compact convex set with $o\in L$, consider the body $L'=L+B\in \mathcal{K}_o^n$. Then, by using \eqref{eq 8821} twice, we have
 	\begin{equation}
 	\label{eq 8830}
 		\begin{aligned}
 			\delta_q (f, 1_{L'})&=\int_\rn h_{L'}(y)d\ac(f;y)+\int_\sn h_{L'}(v)d\widetilde{C}_q^s(f;v)\\
 			&=\int_\rn h_{L}(y)d\ac(f;y)+\int_\sn h_{L}(v)d\widetilde{C}_q^s(f;v)\\
 			& \phantom{qweqwe}+\int_\rn h_{B}(y)d\ac(f;y)+\int_\sn h_{B}(v)d\widetilde{C}_q^s(f;v)\\
 			&= \int_\rn h_{L}(y)d\ac(f;y)+\int_\sn h_{L}(v)d\widetilde{C}_q^s(f;v)+\delta_{q}(f,1_B).
 		\end{aligned}
 	\end{equation}
 	On the other hand, note that $\widetilde{V}_{1,q}(K,L)$ is linear in $L$ with respect to the Minkowski addition. Therefore, by Theorem \ref{theorem delta layer cake}, we have
 	\begin{equation}
 	\label{eq 8831}
 		\delta_q(f,1_{L'}) = \delta_q(f,1_L)+\delta_{q}(f, 1_B).
 	\end{equation}
 	The desired result now follows from combining \eqref{eq 8830} and \eqref{eq 8831}.
\end{proof}

We need the following lemma from \cite{https://doi.org/10.48550/arxiv.2206.13146}.
\begin{lemma}
  \label{lemma pointwise limit}
  	Let $f=e^{-\phi}, g=e^{-\psi}\in \Lc(\rn)$ be such that $f$ has nonzero finite $\L^1$ norm and $g$ is compactly supported. Then for almost all $x\in \rn$, we have
  	\begin{equation}
  		\lim_{t\rightarrow 0^+} \frac{(f\oplus(t\cdot g))(x)-f(x)}{t} = \psi^*(\nabla \phi(x))f(x).
  	\end{equation}
  \end{lemma}

We are now ready to prove the promised variational formula; that is, compute the limit in \eqref{eq 33}.
\begin{theorem}
\label{theorem 31}
 		Let $q>0$ and $f=e^{-\phi}\in \Lc(\rn)$ with non-zero finite $\L^1$ norm. Assume $f$ achieves its maximum at $o$ and
 		\begin{equation}
 		\label{eq 400bb}
 			\limsup_{x\rightarrow o} \frac{|f(x)-f(o)|}{|x|^{\alpha+1}}<\infty,
 		\end{equation}
 		for some $0<\alpha<1$. Let $g=e^{-\psi}\in \Lc(\rn)$ be compactly supported and $g(o)>0$. Then,
 		\begin{equation}
 		\label{eq 40}
 			\delta_q(f,g) = \int_\rn \psi^*(y)d\ac(f;y)+ \int_{\sn} h_{K_g}(v)d\widetilde{C}_q^s(f;v).
 		\end{equation}
 \end{theorem}
 \begin{proof}
 	Let $f_t = f\oplus(t\cdot g)$. Since $g$ is compactly supported, there exists $A>0$ such that $g\leq A 1_{K_g}$ and therefore $f_t\leq \widetilde{f}_t= f\oplus(t\cdot A1_{K_g})$.
 	
 	By Lemma \ref{lemma pointwise limit},
 	\begin{equation}
 		\lim_{t\rightarrow 0^+} \frac{\widetilde{f}_t-f_t}{t} = \lim_{t\rightarrow 0^+} \frac{\widetilde{f}_t-f}{t}-\lim_{t\rightarrow 0^+} \frac{{f}_t-f}{t} = h_{K_g}(\nabla \phi)f+\ln A f-\psi^*(\nabla \phi)f.
 	\end{equation}
 	
 	By Fatou's lemma, we have
 	\begin{equation}
 	\begin{aligned}
 		\liminf_{t\rightarrow 0^+}\int_\rn \frac{\widetilde{f}_t-f_t}{t}|x|^{q-n} dx&\geq \int_\rn [h_{K_g}(\nabla \phi)f+\ln A f-\psi^*(\nabla \phi)f ]\cdot |x|^{q-n} dx\\
 		&=\int_\rn (h_{K_g}(y)-\psi^*(y))d\ac(f;y)+\ln A \cdot \widetilde{V}_q(f)
 	\end{aligned}
 	\end{equation}
 	
 	Since $g(o)>0$, we have $o\in K_g$. By Theorem \ref{theorem 11}, we have
 	\begin{equation}
 	\begin{aligned}
 			&\lim_{t\rightarrow 0^+}\int \frac{\widetilde{f}_t-f}{t}|x|^{q-n} dx \\
 			&= \int_\rn h_{K_g}(y)d\ac(f;y)+\ln A \cdot \widetilde{V}_q(f)+ \int_{\sn} h_{K_g}(v)d\widetilde{C}_q^s(f;v).
 	\end{aligned}
 	\end{equation}
 	Combining the above two formulas, we have
 	\begin{equation}
 		\limsup_{t\rightarrow 0^+}\int \frac{f_t-f}{t}|x|^{q-n}dx\leq \int_{\rn} \psi^*(y)d\ac(f;y)+\int_{\sn} h_{K_g}(v)d\widetilde{C}_q^s(f;v).
 	\end{equation}
 	
 	For the other direction of the inequality, define for each positive integer $j$, the set
 	\begin{equation}
 		Q_j = \left\{x\in \rn: g(x)\geq \frac{1}{j}\right\}.
 	\end{equation}
 	Since $g$ is compactly supported, we conclude that $Q_j\uparrow K_g$. Since $g(o)>0$, for sufficiently large $j$, we have $o\in Q_j$. We focus on such $j$.
 	
 	Let
 	\begin{equation}
 		\bar{g}_j = \frac{1}{j}1_{Q_j}, \quad \text{and } \bar{f}_{j,t} = f\oplus t\cdot \bar{g}_j.
 	\end{equation}
 	Note that $\bar{f}_{j,t}\leq f_t$. Arguing the same way as before, we have
 	\begin{equation}
 		\liminf_{t\rightarrow 0^+}\int \frac{f_t-\bar{f}_{j,t}}{t}|x|^{q-n} dx\geq \int_\rn (\psi^*(y)-h_{Q_j}(y))\,d\ac(f;y)+\ln j \cdot \widetilde{V}_q(f),
 	\end{equation}
 	and
 	\begin{equation}
 	\begin{aligned}
 			&\lim_{t\rightarrow 0^+}\int \frac{\bar{f}_{j,t}-f}{t}|x|^{q-n} dx \\
 			&= \int_\rn h_{Q_j}(y)d\ac(f;y)-\ln j \cdot \widetilde{V}_q(f)+ \int_{\sn} h_{Q_j}(v)d\widetilde{C}_q^s(f;v).
 	\end{aligned}
 	\end{equation}
 	Adding the above two formulas, we have
 	\begin{equation}
 		\liminf_{t\rightarrow 0^+}\int \frac{f_t-f}{t}|x|^{q-n}dx\geq \int_\rn \psi^*(y)\,d\ac(f;y)+ \int_{\sn} h_{Q_j}(v)d\widetilde{C}_q^s(f;v).
 	\end{equation}
 	Letting $j\rightarrow \infty$, by the monotone convergence theorem, we have
 	\begin{equation}
 		\liminf_{t\rightarrow 0^+}\int \frac{f_t-f}{t}|x|^{q-n}dx\geq \int_\rn \psi^*(y)\,d\ac(f;y)+ \int_{\sn} h_{K_g}(v)d\widetilde{C}_q^s(f;v).
 	\end{equation}
 	This completes the proof.
 \end{proof}

 {
 We remark here that the hypothesis \eqref{eq 400} in Theorem \ref{theorem delta layer cake} (and consequently \eqref{eq 400bb} in Theorem \ref{theorem 31}) is not the best hypothesis. In particular, we show that if the level sets of $f$ near the origin are \emph{uniformly} in ``good shape'', then Theorem \ref{theorem delta layer cake} still holds when $L\in \mathcal{K}_o^n$.

For $K\in \mathcal{K}_o^n$, define
\begin{equation}
	r_K = \max\{r\geq 0: rB\subset K\}
\end{equation}
and
\begin{equation}
	R_K = \min\{r\geq 0: K\subset rB\}.
\end{equation}
\begin{prop}
\label{prop 8221}
	Let $f\in \Lc(\rn)$ with non-zero finite $\L^1$ norm and $L\in \mathcal{K}_o^n$. Assume $f$ achieves its maximum at $o$. If there exist $\varepsilon_0>0$ and $C>0$ such that for almost all  $f(o)-\varepsilon_0<s<f(o)$, we have
	\begin{equation}
	\label{eq 8222}
		1\leq \frac{R_{[f\geq s]}}{r_{[f\geq s]}}<C,
	\end{equation}
	 	then
 	\begin{equation}
 		\delta_q (f, 1_L)= \int_{0}^\infty \widetilde{V}_{1,q}([f\geq s],L)\,ds<\infty.
 	\end{equation}
\end{prop}

\begin{proof}
	Denote $f(o)$ by $M$ and $K_s=[f\geq s]$. Like in the proof of Theorem \ref{theorem delta layer cake}, it is sufficient to show
	\begin{equation}
	\label{eq 8221}
 		\lim_{t\rightarrow 0^+} \int_{M-\varepsilon_0}^{M}\frac{ \widetilde{V}_q(K_s+tL)-\widetilde{V}_q(K_s)}{t}ds=\int_{M-\varepsilon_0}^{M}\lim_{t\rightarrow 0^+}\frac{ \widetilde{V}_q(K_s+tL)-\widetilde{V}_q(K_s)}{t}ds,
 	\end{equation}
 	and
 	\begin{equation}
 	\label{eq 8225}
 		\lim_{t\rightarrow 0^+} \int_0^{M-\varepsilon_0}\frac{ \widetilde{V}_q(K_s+tL)-\widetilde{V}_q(K_s)}{t}ds=\int_0^{M-\varepsilon_0}\lim_{t\rightarrow 0^+}\frac{ \widetilde{V}_q(K_s+tL)-\widetilde{V}_q(K_s)}{t}ds.
 	\end{equation}
 	Note that the latter follows in the same way as before. Hence, we only need to justify \eqref{eq 8221}.
 	
 	By \eqref{eq 8222}, it is simple to see that for almost all $s\in (M-\varepsilon_0, M)$ and $t\in (0,1)$, we have
 	\begin{equation}
 		\frac{R_{K_s+tL}}{r_{K_s+tL}}<C_1
 	\end{equation}
 	for some $C_1>0$ independent of $s$ and $t$. Consequently, there exist $c_2>0$ independent of $s$ and $t$ such that
 	\begin{equation}
 	\label{eq 8223}
 		c_2<\nu_{K_s+tL}(x)\cdot \frac{x}{|x|}\leq 1,
 	\end{equation}
 	for almost all $x\in \partial (K_s+tL)$.
 	
 	Let $g(t;s)$ be as defined in the proof of Theorem \ref{theorem delta layer cake}. By the same argument, we have
 	\begin{equation}
 			\frac{g(t;s)-g(0;s)}{t} = \int_{\sn} \frac{h_L(v)}{h_{K_s+\theta L}(v)}d\widetilde{C}_q(K_s+\theta L,v),
 		\end{equation}
 		where $\theta\in [0,t]$ and is dependent on $s$. By the fact that $L\in \mathcal{K}_o^n$ and \eqref{eq 8223}, there exists $C_3>0$ independent of $t$ and $s$ such that
 		\begin{equation}
 		\label{eq 891}
 			\frac{g(t;s)-g(0;s)}{t}\leq C_3 \int_{\sn} \rho_{K_s+\theta L}^{q-1}(u)du.
 		\end{equation}
 		
 		When $q\geq 1$, by \eqref{eq 891} and the fact that $K_s+\theta L\subset K_s+L\subset K_{M-\varepsilon_0}+L$ whenever $t\in (0,1)$ and $s\in (M-\varepsilon_0, M)$, we have
 		\begin{equation}
 			\frac{g(t;s)-g(0;s)}{t}\leq C_3 \int_{\sn} \rho_{K_{M-\varepsilon_0}+L}^{q-1}(u)du\leq C_4,
 		\end{equation}
 		for some $C_4>0$. Equation \eqref{eq 8221} then follows from the bounded convergence theorem.
 		
 		Let us now concentrate on the case $q\in (0,1)$. By the fact that $f$ is positive in a neighborhood of the origin, there exists $c_5>0$ such that
 		\begin{equation}
 			f(x)\geq M/2,
 		\end{equation} 	
 		for all $|x|\leq c_5$. By log-concavity of $f$, we have
 		\begin{equation}
 		\begin{aligned}
 			f(y) &= f\left(\left(1-\frac{|y|}{c_5}\right)o+ \frac{|y|}{c_5}\left(\frac{y}{|y|}c_5\right)\right)\\
 			&\geq M^{1-\frac{|y|}{c_5}}\left(\frac{M}{2}\right)^{\frac{|y|}{c_5}}\\
 			&= M \cdot 2^{-\frac{|y|}{c_5}},
 		\end{aligned}		
 		\end{equation}	
 		for every $0<|y|\leq c_5$. Consequently, there exist $\delta_0>0$ and $C_6>0$  such that for all $s\in (M-\delta_0,M)$, we have
 		\begin{equation}
 			K_s \supset C_6(\log M-\log s)B.
 		\end{equation}
 		Therefore, \eqref{eq 891} and the fact that $0<q<1$ in combination with $K_s\subset K_s+\theta L$, show that  for each $s\in (M-\delta_0,M)$,
 		\begin{equation}
 			\frac{g(t;s)-g(0;s)}{t}\leq C_3 \int_{\sn}\rho_{K_s}^{q-1}(u)du\leq C_7 (\log M-\log s)^{q-1},
 		\end{equation}
 		for some $C_7>0$. Note that
 		\begin{equation}
 			\int_{M-\delta_0}^M(\log M-\log s)^{q-1}ds<\infty.
 		\end{equation}
 		Hence, by the dominated convergence theorem, we conclude the validity of \eqref{eq 8221} with $\varepsilon_0$ replaced by $\delta_0$. Note that \eqref{eq 8225} holds with $\varepsilon_0$ replaced by $\delta_0$ as well---via the exact same argument. Therefore, we derive the desired result.
 		 \end{proof}
 		
 \begin{remark}
 \label{remark 8101}
 	We remark here that since in the proof of Theorem \ref{theorem 31} we required \eqref{eq 400bb} only for the ability to apply Theorem \ref{theorem 11}, by Proposition \ref{prop 8221}, we conclude that with the additional assumption that the origin is an interior point of the support of $g$, Theorem \ref{theorem 31} continues to hold with \eqref{eq 400bb} replaced by \eqref{eq 8222}. In particular, hypothesis \eqref{eq 8222} allows for log-concave functions such as
 	\begin{equation}
 		f(x)=e^{-\|x\|_K},
 	\end{equation}
 	where $K\in\mathcal{K}^{n}_o$. It is of great interest to see whether assumptions like \eqref{eq 400bb} and \eqref{eq 8222} can be dropped altogether.
 \end{remark}}

A few additional remarks are in order:
\begin{enumerate}
	\item Theorem \ref{theorem 31} justifies why we referred to $\ac(f;\cdot)$ and $\widetilde{C}_q^s(f;\dot)$ as the Euclidean and spherical dual curvature measures for log-concave functions. In particular, when $f$ and $g$ are characteristic functions of convex bodies containing the origin in their respective interiors, \eqref{eq 40} recovers its convex geometric counterpart \eqref{eq 92}.
	\item Since $\phi$ is convex and therefore almost everywhere differentiable in the support of $f$, the measure $\ac(f;\cdot)$ is well-defined. Notice that its total measure is equal to $\widetilde{V}_q(f)$. Therefore, by Proposition \ref{prop 12}, as long as $f$ is in $\L^1$, the measure $\ac(f;\cdot)$ is always finite.
	\item Since $f$ is log-concave, its support $K_f$ is necessarily convex and therefore it makes sense to write $\nu_{K_f}$. However, even with the assumption that $f\in \L^1$, it might not be the case that the measure $\widetilde{C}^{s}_{q}(f;\cdot)$ is finite. To see this, simply take the example where $K_f$ has a facet that contains the origin in the relative interior. In that case, the density $|x|^{q-n}$ will have a non-integrable singularity at the origin in that subspace when $q<1$. This counterexample suggests that we must impose some condition on $f$ such that $f$ ``contains the origin in the interior''. With the additional assumption that $K_f$ contains the origin in its interior, one may show that $\widetilde{C}^{s}_{q}(f;\cdot)$ is always a finite measure.

	\item When $f=1_K$ is the characteristic function of a convex body $K\in \mathcal{K}_o^n$, the Euclidean dual curvature measure $\ac(f;\cdot)$ reduces to a point mass concentrated at the origin  and
	\begin{equation}
		\widetilde{C}^{s}_{q}(f;\cdot) = \widetilde{C}_{1,q}(K,\cdot).
	\end{equation}
	\item Let us emphasize again that although the weight function $\omega_q=|x|^{q-n}$ is the only type of weight functions treated in this section, many results presented here can be shown for other weight functions (with only small changes). In particular, the Gaussian weight function is one of the many weight functions that can be used here to replace $\omega_q$. We mention that in the setting of convex bodies, a variational formula for the Gaussian volume and its corresponding Gaussian Minkowski problem was previously studied by the authors in \cite{MR4252759}.
\end{enumerate}

\section{The Minkowski problem for $\ac$}

The purpose of this section is to study the following Minkowski problem for the Euclidean dual curvature measure of log-concave functions.

We will restrict our attention to the even case---the prescribed measure $\mu$ in  \eqref{eq 96} is even and we are restricting our solution set to all even functions $f\in \Lc(\rn)$.

\subsection{The variational structure} In this subsection, we convert the solvability of \eqref{eq 96} into the existence of a minimizer to a minimization problem.

We recall the following result in \cite[Proposition 2.1]{https://doi.org/10.48550/arxiv.2006.16933}. 

\begin{prop}[\cite{https://doi.org/10.48550/arxiv.2006.16933}] \label{prop 881}
	Let $\phi, \zeta:\rn \rightarrow (-\infty, \infty]$ be lower semi-continuous functions. Assume $\zeta$ is bounded from below and $\zeta(0), \phi(0)<\infty$. Then at every point $x_0\in \rn$ where $\phi^*$ is differentiable, we have
	\begin{equation}
		\left.\frac{\partial}{\partial t}\right|_{t=0^+} (\phi+t\zeta)^*(x_0) = -\zeta(\nabla \phi^*(x_0)).
	\end{equation}
\end{prop}
Note that in \cite[Proposition 2.1]{https://doi.org/10.48550/arxiv.2006.16933}, the derivative is stated as a one-sided derivative. However, for $\zeta\in C_c(\rn)$, one can consider $-\zeta$ and immediately get that the derivative is double-sided.

We will require the following variational formula.
\begin{lemma}
\label{lemma 51}
	Let $f\in \Lc(\rn)$ be an even function and $q>0$. Suppose $f$ has non-zero finite $(q-n)$-th moment and takes the form $f=e^{-\phi}$ for some $\phi\in \cvx(\rn)$. For  $\zeta\in C_c(\rn)$, define
	\begin{equation}
		\phi_t(x) = (\phi^*+t\zeta)^*(x),
	\end{equation}
	and
	\begin{equation}
		f_t(x)=e^{-\phi_t(x)}.
	\end{equation}
	Then, we have
	\begin{equation}
	\label{eq local 1}
		\left.\frac{d}{dt}\right|_{t=0}\widetilde{V}_q(f_t) = \int_\rn \zeta(\nabla \phi(x))|x|^{q-n} f(x)\,dx = \int_\rn \zeta(y)d\ac(f;y).
	\end{equation}
\end{lemma}
\begin{proof}
	We note that the second equality follows directly from the definition of $\ac(f;\cdot)$ and therefore only the first equality needs a proof.
	
	By the fact the $f$ is even, Proposition \ref{prop 881}, and the remark immediately below it, we have
	\begin{equation}
		\left.\frac{\partial \phi_t(x)}{\partial t}\right|_{t=0} = -\zeta(\nabla \phi(x)),
	\end{equation}
	almost everywhere in $\rn$.
	
	For simplicity, write
	\begin{equation}
		h_t(x) = |x|^{q-n} f_t(x).
	\end{equation}
	Since $\zeta\in C_c(\rn)$, there exists $M>0$ such that $|\zeta|\leq M$. Therefore, we have
	\begin{equation}
		\phi^*-|t|M\leq \phi^*+t\zeta\leq \phi^*+|t|M.
	\end{equation}
	Since Legendre transform reverses the order, we have
	\begin{equation}
		\phi-|t|M=(\phi^*+|t|M)^*\leq (\phi^*+t\zeta)^*\leq (\phi^*-|t|M)^*=\phi +|t|M.
	\end{equation}
	Using the above estimates, we have the existence of $C>0$, such that
	\begin{equation}
	\begin{aligned}
		\left|\frac{h_t(x)-h(x)}{t}\right|&= |x|^{q-n} \left|\frac{e^{-\phi_t(x)}-e^{-\phi(x)}}{t}\right|\\&\leq |x|^{q-n} e^{-\phi(x)} \left|\max\left\{\frac{e^{|t|M}-1}{t}, \frac{e^{-|t|M}-1}{t}\right\}\right|\\
		&\leq 2C |x|^{q-n} e^{-\phi(x)},
	\end{aligned}
	\end{equation}
	for sufficiently small $|t|$. Therefore, by the dominated convergence theorem, we have
	\begin{equation}
	\begin{aligned}
				\left.\frac{d}{dt}\right|_{t=0}\widetilde{V}_q(f_t) &= \int_\rn \left.\frac{\partial }{\partial t}\right|_{t=0} h_t(x)dx\\
				& =\int_\rn \zeta(\nabla \phi(x))|x|^{q-n}f(x)dx.
	\end{aligned}
	\end{equation}
\end{proof}

Let $A>0$ be an arbitrary number and $\mu$ be a nonzero even finite Borel measure on $\rn$. Consider the following optimization problem
\begin{equation}
\label{eq local 4}
	\inf \left\{\int_\rn \phi d\mu: \widetilde{V}_q(e^{-\phi^*})\geq  A, \phi \geq 0, \phi \in \L^1(\mu), \phi \text{ is an even function} \right\}.
\end{equation}

The following lemma shows that the Euler-Lagrange equation of the above constrained optimization problem implies the existence of a solution to \eqref{eq 96}.

\begin{lemma}
\label{lemma 13}
	Let $q>0$ and $\mu$ be a non-zero even finite Borel measure on $\rn$ that is not concentrated in any proper subspace. If an even function $\phi_0\in \cvx(\rn)$ is such that $\phi_0\in \L^1(\mu)$, $\phi_0(o)>0$,
	\begin{equation}
		\widetilde{V}_q(e^{-\phi_0^*})= A, \qquad\text{ for some } A>0,
	\end{equation}
	and
	\begin{equation}
		\int_\rn \phi_0d\mu=\inf \left\{\int_\rn \phi d\mu: \widetilde{V}_q(e^{-\phi^*})\geq A, \phi \geq 0, \phi \in \L^1(\mu),\phi \text{ is an even function} \right\},
	\end{equation}
	then
	\begin{equation}
		\mu = \ac(f_0;\cdot),
	\end{equation}
	where $$f_0=\frac{|\mu|}{\widetilde{V}_q(e^{-\phi_0^*})}e^{-\phi_0^*}.$$
Moreover, $f_0\in \L^1$.
\end{lemma}
\begin{proof}
	Let $\zeta\in C_c(\rn)$ be an even function. Set
	\begin{equation}
		\phi_t(x) = \phi_0(x)+t\zeta(x)
	\end{equation}
	and
	\begin{equation}
		\lambda(t) =\widetilde{V}_q(e^{-\phi_t^*}).
	\end{equation}
	The fact that $\lambda(t)$ is finite for small $|t|$ comes from Proposition \ref{prop 51}, the fact that $\phi_0\in \L^1(\mu)$, and that $\mu$ is a finite measure not concentrated in any proper subspaces.
	Set
	\begin{equation}
		\widetilde{\phi}_t(x) = \phi_t(x)-\ln \lambda(t)+\ln A.
	\end{equation}
	It is simple to notice that
	\begin{equation}
		\widetilde{V}_q(e^{-\widetilde{\phi}_t^*})=\widetilde{V}_q(e^{-\phi_t^*})\frac{A}{\lambda(t)}= A.
	\end{equation}
	
	Since $\zeta\in C_c(\rn)$, there exists $M>0$ such that $|\zeta|<M$. This implies that
	\begin{equation}
		\lambda(t) \leq \widetilde{V}_q(e^{-\phi_0^*}) e^{|t|M} = Ae^{|t|M}.
	\end{equation}
	Thus, the choice of $\widetilde{\phi}_t$ implies that when $|t|$ is sufficiently small, we have
	\begin{equation}
		\widetilde{\phi}_t \geq \phi_0(o) -2|t|M> 0.
	\end{equation}
	Such $\widetilde{\phi}_t$ satisfies the constraints of the optimization problem. Since $\phi_0$ is the minimizer and   $\widetilde{\phi}_0 = \phi_0$, by Lemma \ref{lemma 51},
	\begin{equation}
	\begin{aligned}
		0 &= \left.\frac{d}{dt}\right|_{t=0} \int_\rn \widetilde{\phi}_td\mu\\
			&= \int_\rn \zeta(x)d\mu(x) - \frac{|\mu|}{A} \lambda'(0)\\
			& = \int_\rn \zeta(x)d\mu(x) - \int_\rn \zeta(x)d\ac(f_0;x),
	\end{aligned}
	\end{equation}
	where $f_0=e^{-\phi_0^*+\ln |\mu|-\ln(A)}$. Since $\zeta$ is arbitrary in $C_c(\rn)$, this implies that
	\begin{equation}
		\mu =  \ac(f_0; \cdot).
	\end{equation}
	
	To see that $f_0\in \L^1$, note that since $\mu$ is not concentrated in any proper subspaces and $\phi_0\in \L^1(\mu)$, we conclude that $\phi_0$ is finite in a neighborhood of the origin. The fact that $f_0\in \L^1$ now follows from the Proposition \ref{prop 51} with $q=n$.
\end{proof}

Note that although the requirement $\phi\geq 0$ in the constraints of the optimization problem is a closed condition, to make the Euler Lagrange equation equal to $0$, we have to establish that the minimizer actually satisfies a stronger condition ($ \phi_0>0)$. This will be done in the next subsection.

\subsection{Existence of a minimizer} This section is dedicated to showing the existence of a minimizer to \eqref{eq local 4} under the assumption that $\mu$ is an even measure.

It turns out that the $C^0$ estimates here are closely related to the estimates of the dual quermassintegrals of level sets of convex functions. This is perhaps not surprising, given that in the case of characteristic functions, the $(q-n)$-th moment of a log-concave function reduces to the $q$-th dual quermassintegral of a convex body. The following lemma reveals the simple fact that if the images of the orthogonal projections of a set of convex bodies onto a $k$-dimensional subspace, where $k=1,\dots, n-1$, are uniformly bounded, then their $q$-th dual quermassintegrals are uniformly bounded when $q<k$.

\begin{lemma}
\label{lemma 2}
	Let $k=1,\dots, n-1$, and $0<q<k$. For each $R>0$, there exists $C>0$,  such that for all $K\in \mathcal{K}^n_o$ satisfying
	\begin{equation}
	\label{eq 8181}
		P_\xi K \subset B(R)\cap \xi, \quad\text{ for some $k$-dim subspace }\xi\subset \rn,
	\end{equation}
we have
	\begin{equation}
		\widetilde{V}_q(K)<C.
	\end{equation}
Here, we use $P_\xi K$ to denote the image of the orthogonal projection of $K$ onto $\xi$.
\end{lemma}
\begin{proof}
	By \eqref{eq 8181}, we have
	\begin{equation}
		K\subset B(\sqrt{2}R)\cup [P_{\xi}K \times(\xi^{\perp}\setminus B(R))]\subset B(\sqrt{2}R)\cup [(B(R)\cap \xi) \times(\xi^{\perp}\setminus B(R))].
	\end{equation}
	Hence,
	\begin{equation}
	\begin{aligned}
		\widetilde{V}_q(K) &\leq \int_{\xi^\perp\setminus B(R)}\int_{B(R)\cap \xi} |(y,z)|^{q-n}d\mathcal{H}^k(y) d\mathcal{H}^{n-k}(z)+ \int_{B(\sqrt{2}R)}|x|^{q-n}dx\\
		&\leq \int_{\xi^\perp\setminus B(R)}|z|^{q-n}d\mathcal{H}^{n-k}(z)\int_{B(R)\cap \xi} d\mathcal{H}^k(y) + \frac{C}{q}(\sqrt{2}R)^q\\
		&= CR^k \int_{R}^{\infty} \rho^{q-n}\rho^{n-k-1}d\rho + \frac{C}{q}(\sqrt{2}R)^q\\
		&= C \frac{1}{k-q}R^{q} + \frac{C}{q}(\sqrt{2}R)^q.
	\end{aligned}
	\end{equation}
\end{proof}

For $K\in \mathcal{K}^n_o$, we write
\begin{equation}
	\overline{V}_{q}(K) = \left(\frac{1}{q}\int_\sn \rho_K^q(u)du\right)^{\frac{1}{q}},
\end{equation}
for the normalized version of dual quermassintegrals.

A quick application of Lemma \ref{lemma 2} gives the following Blaschke-Santal\'{o} type estimates for normalized dual quermassintegrals. This is a special case of Theorem 3.1 in \cite{MR4259871}.
\begin{lemma}
\label{lemma 3}
	Let $K$ be an origin-symmetric convex body in $\rn$. If $q\in (0,1)$ and $p>0$, then there exists $C>0$ independent of $K$ such that
	\begin{equation}
	\label{eq local 32}
		\overline{V}_{q}(K^*) \overline{V}_{p}(K)<C.
	\end{equation}
\end{lemma}
\begin{proof}
Let $v_0\in \sn$ be such that $$h_{K^*}(v_0) = \min_{v\in \sn} h_{K^*}(v).$$
Note that the functional $\overline{V}_{q}(K^*) \overline{V}_{p}(K)$ is invariant under rescaling of $K$. Therefore, by rescaling, we may assume $h_{K^*}(v_0)=1$.
This, by the choice of $v_0$, implies
\begin{equation}
	B\subset K^*,
\end{equation}
and consequently,
\begin{equation}
\label{eq 8182}
	K\subset B.
\end{equation}
Moreover,
\begin{equation}
	P_{\xi} K^*\subset B\cap \xi,
\end{equation}
where $\xi = \Span \{v_0\}$. Equation \eqref{eq local 32} follows from  Lemma \ref{lemma 2} and \eqref{eq 8182}.

\end{proof}

By integrating (in a certain way) over level sets of a log-concave function, Lemma \ref{lemma 3} readily implies the following Blaschke-Santal\'{o} type inequality for log-concave functions.

\begin{lemma}
\label{lemma 4}
	Let $\phi\in \cvx(\rn)$ be an even function with $\phi(o)=0$. Assume $\phi$ is finite in a neighborhood of the origin and $\lim_{|x|\rightarrow \infty}\phi(x)=\infty$. Suppose $q>0$ and $0<p<1$. There exists $C>0$, independent of $\phi$, such that
	\begin{equation}
		\left(\int_\rn |x|^{q-n} e^{-\phi^*(x)}dx \right)\left(\int_0^\infty e^{-t}\overline{V}_{p}([\phi\leq t])^{q}dt\right)<C.
	\end{equation}
\end{lemma}
\begin{proof}
Recall that, for convenience, when no confusion arises, a constant $C>0$ might change from line to line (or even within the same line).

By Proposition \ref{prop 51}, since $\phi$ is finite in a neighborhood of the origin, we have
\begin{equation}
	\int_\rn |x|^{q-n} e^{-\phi^*(x)}dx<\infty.
\end{equation}

	By the definition of the Legendre-Fenchel transform, we have
	\begin{equation}
		\phi^*(x) = \sup \{\langle x, y\rangle - \phi(y)\}\geq -\phi(o)=0.
	\end{equation}
	Furthermore, we have for any $s,t\geq 0$,
	\begin{equation}
		[\phi^*\leq s]\subset (s+t) [\phi\leq t]^*.
	\end{equation}
	The proof of this fact can be found, for example, in the proof of Theorem 2.1 in \cite{MR2220210}. 
	Note that since $\phi$ is finite in a neighborhood of the origin and $\phi(o)=0$, we conclude that for each $t>0$, the sublevel set $[\phi\leq t]$  contains the origin in the interior. On the other hand, since $\lim_{|x|\rightarrow \infty}\phi(x)=\infty$, the level set $[\phi\leq t]$ is bounded. Lemma \ref{lemma 3} now implies the existence of $C>0$ such that
\begin{equation}
	\overline{V}_{q}([\phi^*\leq s])\leq (s+t) \overline{V}_{q}([\phi\leq t]^*)\leq C(s+t) \overline{V}_{p}([\phi\leq t])^{-1}.
\end{equation}
By definition of $\overline{V}_{q}$, we have
\begin{equation}
\label{eq local 5}
	\int_{[\phi^*\leq s]} |x|^{q-n}dx \leq C (s+t)^{q} \overline{V}_{p}([\phi\leq t])^{-q}.
\end{equation}

Set
\begin{equation}
	F(s) = \begin{cases}
		e^{-s}\int_{[\phi^*\leq s]} |x|^{q-n}dx, &\text{ if } s\geq 0\\
		0, &\text{ otherwise,}
	\end{cases}
\end{equation}
and
\begin{equation}
	G(t) = \begin{cases}
		e^{-t}\overline{V}_{p}([\phi\leq t])^{q}, &\text{ if } t\geq 0\\
		0, &\text{ otherwise,}
	\end{cases}
\end{equation}
and
\begin{equation}
	H(x) = \begin{cases}
		\sqrt{C}e^{-x}(2x)^{\frac{q}{2}}, &\text{ if } x\geq 0,\\
		0, &\text{ otherwise.}
	\end{cases}
\end{equation}

Then, \eqref{eq local 5} implies that for any $s,t\geq 0$,
\begin{equation}
	H\left(\frac{1}{2}s +\frac{1}{2}t\right)= \sqrt{C}e^{-\frac{s+t}{2}}\left(s+t\right)^{\frac{q}{2}}\geq F\left(s\right)^{\frac{1}{2}}G\left(t\right)^{\frac{1}{2}}.
\end{equation}
It is simple to check that the above inequality is also true when $s$ or $t$ is negative, in which case the right-hand side of the inequality is $0$.

Therefore, by the Pr\'ekopa-Leindler inequality, we have
\begin{equation}
\label{eq local 6aa}
	\left(\int_0^\infty e^{-s}\int_{[\phi^*\leq s]} |x|^{q-n}dx ds\right) \left(\int_0^\infty e^{-t}\overline{V}_{p}([\phi\leq t])^{q}dt\right) \leq C \left(\int_{0}^\infty e^{-x} x^{\frac{q}{2}}dx\right)^2<C.
\end{equation}

The fact that $\phi^*\geq 0$ and layer-cake representation now imply
\begin{equation}
	\int_\rn |x|^{q-n} e^{-\phi^*(x)}dx= \int_{0}^\infty \int_{[e^{-\phi^*}\geq t]} |x|^{q-n}dxdt = \int_{0}^\infty e^{-s}\int_{[\phi^*\leq s]} |x|^{q-n} dxds.
\end{equation}
This, when combined with \eqref{eq local 6aa}, gives the desired estimate.
\end{proof}

The above lemma immediately implies the following comparison.
\begin{lemma}
\label{lemma 11}
	Let $\phi\in \cvx(\rn)$ be an even function with $\phi(o)=0$ and $\mu$ be a nonzero even finite Borel measure not concentrated in any proper subspace. Suppose $q>0$ and $\phi\in \L^1(\mu)$. There exist $C_1>0$ and $C_2<0$, independent of $\phi$ such that
	\begin{equation}
	\label{eq local 6}
		\int_\rn \phi d\mu \geq C_1 \widetilde{V}_q^{\frac{1}{q}}(e^{-\phi^*}) +C_2.
	\end{equation}
\end{lemma}

\begin{proof}
If
\begin{equation*}
	\widetilde{V}_q(e^{-\phi^*})=0,
\end{equation*}
or equivalently $\phi^*$ is almost everywhere $+\infty$, there is nothing to prove. Therefore, we may assume that $\phi^*$ is finite in a neighborhood of the origin, and by Proposition \ref{prop 51}, we have that
\begin{equation*}
	\int_\rn e^{-\phi(x)}dx<\infty,
\end{equation*}
or equivalently, $\phi\rightarrow \infty$ as $|x|\rightarrow \infty$.

	We first note that since $\mu$ is not concentrated in any proper subspaces, there exists $c_1>0$ such that 
	\begin{equation}
		\int_\rn |\langle x, u\rangle|d\mu(x)>c_1,
	\end{equation}
	for every $u\in \sn$.
	
	Set
	\begin{equation}
		K  = \{x\in \rn: \phi(x)\leq 1\}.
	\end{equation}
	Since $\phi\in \L^1(\mu)$ is even and $\mu$ is not concentrated in any proper subspace, we have that $\phi$ is finite in a neighborhood of $o$. Therefore, by the fact that $\phi(o)=0$, we have that $K$ is a symmetric closed convex set that contains the origin in its interior.
	
	Let $r_K$ be such that
	\begin{equation}
		r_K = \sup\{r>0: rB\subset K\}.
	\end{equation}
	The facts that $\phi\rightarrow \infty$ when $|x|\rightarrow \infty$ imply that $K\neq \rn$. Therefore $0<r_K<\infty$. Note that since $K$ is closed, there exists $u_0\in \sn$ such that $r_K u_0\in \partial K$. This in turn implies that $h_K(u_0)=r_K$.
		
	We fix now an arbitrary $x\in \rn$ with $|\langle x,u_0\rangle|>2r_K$. We consider $x' = \frac{2r_K}{|\langle x,u_0\rangle|}x$. Note that by choice of $x$, we have $|\langle x',u_0\rangle|>r_K$. This implies that $x'\notin K$ and therefore $\phi(x')>1$. By convexity of $\phi$, we now have
	\begin{equation}
		1<\phi(x') = \phi\left(\left(1-\frac{2r_K}{|\langle x,u_0\rangle|}\right)o+\frac{2r_K}{|\langle x, u_0\rangle|}x\right)\leq  \frac{2r_K}{|\langle x,u_0\rangle|} \phi(x).
	\end{equation}
	Hence, for every $x\in \rn$, we have
	\begin{equation}
		\phi(x)+1 \geq \frac{1}{2r_K} |\langle x, u_0\rangle|.
	\end{equation}
	Integrating with respect to $\mu$, we have
	\begin{equation}
	\label{eq local 7}
		\int_\rn \phi d\mu \geq \frac{1}{2r_K}\int_\rn |\langle x, u_0\rangle| d\mu(x)-|\mu|\geq \frac{c_1}{2r_K}-|\mu|.
	\end{equation}
	
	We now estimate the right-hand side of \eqref{eq local 6}. By Lemma \ref{lemma 4}, there exist $0<p<1$ and $c_2>0$ such that
	\begin{equation}
	\label{eq local 8}
		\widetilde{V}_q(e^{-\phi^*})< c_2 \left(\int_0^\infty e^{-t}\overline{V}_{p}([\phi\leq t])^{q}dt\right)^{-1}.
	\end{equation}
	
	Note that by the convexity of $\phi$ and since $\phi(o)=0$, we have
	\begin{equation}
		[\phi\leq t]\supset t K,
	\end{equation}
	for $t<1$. Therefore, we have
	\begin{equation}
	\label{eq local 9}
	\begin{aligned}
		\int_0^\infty e^{-t}\overline{V}_{p}([\phi\leq t])^{q}dt&\geq \int_{0}^1 e^{-t} \overline{V}_{p}(tK)^{q}dt \\
		& = \overline{V}_{p}(K)^{q} \int_{0}^1 e^{-t}t^{q}dt\\
		&\geq \overline{V}_{p}(r_KB)^{q}\int_{0}^1 e^{-t}t^{q}dt\\
		&= c_3 r_K^{q},
	\end{aligned}
	\end{equation}
	for some $c_3>0$.
	
	Combining \eqref{eq local 8} and \eqref{eq local 9}, for some $c_4>0$, we have
	\begin{equation}
		\widetilde{V}_q^{\frac{1}{q}}(e^{-\phi^*})\leq \frac{c_4}{r_K}.
	\end{equation}
 By \eqref{eq local 7}, we have the existence of $c_5>0$ such that
	\begin{equation}
		\int \phi d\mu \geq c_5\widetilde{V}_q^{\frac{1}{q}}(e^{-\phi^*})-|\mu|=:C_1\widetilde{V}_q^{\frac{1}{q}}(e^{-\phi^*})+C_2,
	\end{equation}
	for some $C_1>0$ and $C_2<0$.
	\end{proof}

In \cite[Lemma 17]{MR3341966}, Cordero-Erausquin and Klartag demonstrated that, if $\phi_k$ is a sequence of nonnegative convex functions with uniform $\L^1(\mu)$ bound and $\phi_k(o)=0$, then one may construct a subsequence $\phi_{k_j}$ and a nonnegative convex function $\phi$ such that the $\L^1(\mu)$ norm of $\phi$ is bounded from above by the lower limit of the $\L^1(\mu)$ norm of the subsequence while $\widetilde{V}_n(e^{-\phi^*})$ is bounded from below by the upper limit of $\widetilde{V}_n(e^{-\phi_{k_j}^*})$. As observed by Rotem \cite{https://doi.org/10.48550/arxiv.2006.16933}, the assumption $\phi_k(o)=0$ is only used to know that $\phi_k(\lambda x)$ is increasing in $\lambda$ on $(0,1)$ and this is trivially true when $\phi_k$ is even. Upon further inspection of the proof, it is not hard to see that such a ``selection theorem'' holds for any $q>0$. We state the following generalized version without providing a proof.

\begin{lemma}[\cite{MR3341966}]\label{lemma selection}
	Let $q>0$ and $\mu$ be a non-zero even finite Borel measure on $\rn$. Assume $\mu$ is not concentrated in any proper subspace. If $\phi_k\in \cvx(\rn)$ is non-negative, even and
	\begin{equation}
		\sup_k \int_\rn \phi_k d\mu<\infty,
	\end{equation}
	then, there exists a subsequence $\phi_{k_j}$ and a non-negative, even, convex function $\phi\in \L^1(\mu)$ such that
	\begin{equation}
	\label{eq 62}
		\int_\rn \phi d\mu\leq \liminf_{j\rightarrow\infty} \int_\rn \phi_{k_j}d\mu, \quad \text{and}\quad \widetilde{V}_q(e^{-\phi^*})\geq \limsup_{j\rightarrow\infty} \widetilde{V}_q(e^{-\phi_{k_j}^*}).
	\end{equation}
\end{lemma}
	
We are now ready to show the existence of a minimizer to \eqref{eq local 4}.
\begin{lemma}
\label{lemma 12}
	Let $q>0$ and $\mu$ be a non-zero even finite Borel measure on $\rn$. Suppose $\mu$ is not concentrated in any proper subspace and $\int_\rn |x|d\mu(x)<\infty$. For sufficiently large $A>0$, there exists an even function $\phi_0 \in \cvx(\rn)$ such that $\phi_0\in \L^1(\mu)$, $\phi_0(o)>0$,
	\begin{equation}
		\widetilde{V}_q(e^{-\phi_0^*})= A,
	\end{equation}
	and
	\begin{equation}
	\label{eq local 14}
		\int_\rn \phi_0d\mu=\inf \left\{\int_\rn \phi d\mu: \widetilde{V}_q(e^{-\phi^*})\geq  A, \phi \geq 0, \phi \in \L^1(\mu),\phi \text{ is an even function} \right\}.
	\end{equation}
\end{lemma}

\begin{proof}
Let $c_n>0$ be such that
$$\int_{B(c_n)} |x|^{q-n}dx =1.$$
Set
\begin{equation}
	C = c_n\int_\rn |x|d\mu>0.
\end{equation}
By the condition on $\mu$, it is simple to see that $C$ is finite.

We choose a fixed $A>0$ such that
	\begin{equation}
	\label{eq 8191}
		|\mu|\ln A+ C<C_1A^{\frac{1}{q}}+C_2,
	\end{equation}
	where $C_1, C_2$ are from Lemma \ref{lemma 11} and depend only on $\mu$ and $q$. It is simple to see that such an $A$ exists and in fact, all sufficiently large $A>0$ satisfies \eqref{eq 8191}.
	
	Let $\phi_k$ be a minimizing sequence; that is, we have $\phi_k\geq 0$, $\phi_k\in \L^1(\mu)$, $\phi_k$ is even,
	\begin{equation}
		\widetilde{V}_q(e^{-\phi_k^*}) \geq A,
	\end{equation}
	and
	\begin{equation}
		\lim_{k\rightarrow \infty} \int_\rn \phi_k d\mu =\inf \left\{\int_\rn \phi d\mu: \widetilde{V}_q(e^{-\phi^*})\geq  A, \phi \geq 0, \phi \in \L^1(\mu),\phi \text{ is an even function} \right\}.
	\end{equation}
	
	Notice that
	\begin{equation}
		\phi_k^{**}\leq \phi_k,
	\end{equation}
	and since $\phi_k\geq 0$, we have $\phi_k^{**}\geq 0$. Moreover, note that $(\phi_k^{**})^*=\phi_k^*$. Therefore, we may without loss of generality, assume that $\phi_k \in \cvx(\rn)$.
	
	We claim now that
\begin{equation}
\label{eq 61}
	\sup_{k} \int_\rn \phi_k d\mu<\infty.
\end{equation}
To see this, set
\begin{equation}
	\Gamma(x) = \ln A + c_n|x|,
\end{equation}
It is simple to compute that
	\begin{equation}
		\Gamma^*(x) = 1^\infty_{B(c_n)}(x)-\ln A,
	\end{equation}
	where
	\begin{equation}
		1^\infty_{B(c_n)}(x) = \begin{cases}
			0, &\text{ if }x\in B(c_n),\\
			\infty, &\text{otherwise}.
		\end{cases}
	\end{equation}
	Moreover,
	\begin{equation}
		\widetilde{V}_q(e^{-\Gamma^*})=A,
	\end{equation}
	according to the choice of $c_n$. Note that $\Gamma(x)$ is positive and even. Moreover,
	\begin{equation}
	\label{eq local 15}
		\int_\rn \Gamma(x)d\mu = |\mu|\ln A+ C,
	\end{equation}
	Since $\phi_k$ is a minimizing sequence, we conclude \eqref{eq 61}.
	
	By Lemma \ref{lemma selection}, there exists a subsequence $\phi_{k_j}$ and a non-negative even convex function $\phi_0$ such that \eqref{eq 62} holds. In particular, this suggests that $\phi_0$ is a minimizer to \eqref{eq local 14}.  By possibly replacing $\phi_0$ by $\phi_0^{**}$, we may assume that $\phi_0\in \cvx(\rn)$.
	
	It remains to show that $\phi_0(o)>0$ and
	\begin{equation}
	\label{eq 63}
		\widetilde{V}_q(e^{-\phi_0^*})= A.
	\end{equation}
	
	To see the former, we argue by contradiction. Assume $\phi_0(o)=0$. Therefore, we may conclude by using Lemma \ref{lemma 11} that
	\begin{equation}
		|\mu|\ln A+ C\geq \int_\rn \phi_0 d\mu \geq C_1 \widetilde{V}_q^{\frac{1}{q}}(e^{-\phi_0^*}) +C_2\geq C_1A^{\frac{1}{q}}+C_2.
	\end{equation}
	This is a contradiction to \eqref{eq 8191}.
	
	To show \eqref{eq 63}, if it was not the case, we set
	\begin{equation}
		\widetilde{\phi}_0 = \phi_0-\varepsilon.
	\end{equation}
	Note that for sufficiently small $\varepsilon>0$, we have $\widetilde{\phi}_0>0$ thanks to $\phi_0(o)>0$. Moreover,
	\begin{equation}
		\widetilde{V}_q(e^{-\widetilde{\phi}_0^*})= e^{-\varepsilon}\widetilde{V}_q(e^{-\phi_0^*}) >A,
	\end{equation}
	for sufficiently small $\varepsilon>0$. However, it is trivial to see
	\begin{equation}
		\int_\rn \widetilde{\phi}_0d\mu<\int_\rn {\phi}_0d\mu,
	\end{equation}
	which contradicts the minimality of $\phi_0$.
\end{proof}

Lemmas \ref{lemma 13} and \ref{lemma 12} now immediately solve the Minkowski problem \eqref{eq 96}.
\begin{theorem}
\label{thm 500}
	Let $q>0$ and $\mu$ be a non-zero even finite Borel measure on $\rn$. Suppose $\mu$ is not concentrated in any proper subspace and $\int_\rn |x|d\mu(x)<\infty$. There exists an even $f_0\in \Lc(\rn)$ with nonzero finite $\L^1$ norm such that
	\begin{equation}
		\mu = \ac(f_0;\cdot).
	\end{equation}
\end{theorem}

To complete this section, we show that the assumption that
\begin{equation}
	\int_\rn |x|d\mu(x)<\infty
\end{equation}
is necessary in Theorem \ref{thm 500}. We require the following basic lemma about log-concave functions.
\begin{lemma}
\label{lemma log concave derivative control}
	Suppose $f: \mathbb{R}\rightarrow [0,\infty)$ is a log-concave function. If
	\begin{equation}
	\label{eq 2001}
		\lim_{t\rightarrow \pm \infty} f(t)=0,
	\end{equation}
	then, for each $t_0\geq 0$, we have
	\begin{equation}
		\int_{|t|\geq t_0} |f'(t)|dt\leq 4\sup_{|t|\geq t_0} f(t).
	\end{equation}
\end{lemma}
\begin{proof}
	Since $f$ is log-concave, it is locally Lipschitz in the interior of the interval in which it is positive. Moreover, it is unimodal. Therefore, by \eqref{eq 2001} and the fundamental theorem of calculus, we have
	\begin{equation}
		\int_{t\geq t_0} |f'(t)|dt\leq 2 \sup_{t\geq t_0} f(t),
	\end{equation}
	and
	\begin{equation}
		\int_{t\leq -t_0} |f'(t)|dt\leq 2 \sup_{t\leq -t_0} f(t).
	\end{equation}
	Combining the above two inequality gives us the desired result.
\end{proof}
\begin{theorem}
\label{thm necessity}
	Let $q>0$ and $f\in \Lc(\rn)$ be with nonzero finite $\L^1$ norm. If $K_f$, the support of $f$, contains the origin as an interior point, then
	\begin{equation}
	\label{eq 1002}
		\int_{\rn} |x|  d\ac(f;x) = \int_{\rn} |\nabla f(x)| \cdot |x|^{q-n}dx<\infty.
	\end{equation}
	In particular, \eqref{eq 1002} is valid for even $f\in \Lc(\rn)$ with nonzero finite $\L^1$ norm.
\end{theorem}
\begin{proof}
We first consider the case where $q\in (0,n]$.

Since $K_f$ contains the origin as an interior point, the function $f$ is Lipschitz in $B(r_0)$ for some $r_0>0$. Denote the Lipschitz constant of $f$ inside $B(r_0)$ by $\Lambda>0$. This implies that inside $B(r_0)$, we have $|\nabla f|\leq \Lambda$ almost everywhere. Therefore,
\begin{equation}
	\int_{B(r_0)} |\nabla f(x)| \cdot |x|^{q-n}dx\leq \Lambda\int_{B(r_0)}|x|^{q-n}dx<\infty.
\end{equation}

Therefore, to show \eqref{eq 1002}, we only need to show
\begin{equation}
		\int_{B(r_0)^c} |\nabla f(x)| \cdot |x|^{q-n}dx<\infty.
\end{equation}
In particular, since 
\begin{equation}
	\begin{aligned}
		|\nabla f(x)| = \sqrt{\left(\frac{\partial f}{\partial x_1}(x)\right)^2+\cdots + \left(\frac{\partial f}{\partial x_n}(x)\right)^2}\leq \sum_{i=1}^n \left|\frac{\partial f}{\partial x_i}(x)\right|,
	\end{aligned}
\end{equation}
without loss of generality, it is sufficient to show
	\begin{equation}
		\int_{B(r_0)^c} \left|\frac{\partial f}{\partial x_n}(x)\right|\cdot |x|^{q-n}dx<\infty.
	\end{equation}
	We write $x=(y,t)\in \mathbb{R}^{n-1}\times \mathbb{R}$. Note that since $f\in \L^1$, then $f\rightarrow 0$ as $|x|\rightarrow \infty$. Therefore, since $0<q\leq n$, by Lemma \ref{lemma log concave derivative control},
	\begin{equation}
	\label{eq 1000}
	\begin{aligned}
		&\int_{B(r_0)^c} \left|\frac{\partial f}{\partial t}(x)\right| \cdot |x|^{q-n}dx\\
		\leq& \left(\int_{B(r_0/2)^c\cap \mathbb{R}^{n-1}}\int_{-\infty}^{\infty}+\int_{B(r_0/2)\cap \mathbb{R}^{n-1}}\int_{|t|\geq \frac{\sqrt{3}}{2}r_0}\right) \left|\frac{\partial f}{\partial t}(y,t)\right||(y,t)|^{q-n}dtdy\\
		\leq &4\int_{B(r_0/2)^c\cap \mathbb{R}^{n-1}} \sup_{t\in \mathbb{R}} f(y,t)\cdot |y|^{q-n}dy+ \left(\frac{\sqrt{3}}{2}r_0\right)^{q-n} \int_{B(r_0/2)\cap \mathbb{R}^{n-1}}\int_{|t|\geq \frac{\sqrt{3}}{2}r_0}\left|\frac{\partial f}{\partial t}(y,t)\right|dtdy\\
		\leq & 4\int_{B(r_0/2)^c\cap \mathbb{R}^{n-1}} \sup_{t\in \mathbb{R}} f(y,t)\cdot |y|^{q-n}dy + 4\left(\frac{\sqrt{3}}{2}r_0\right)^{q-n}\int_{B(r_0/2)\cap \mathbb{R}^{n-1}}\sup_{t\in \mathbb{R}}f(y,t)dy\\
		\leq &4\int_{B(r_0/2)^c\cap \mathbb{R}^{n-1}} \sup_{t\in \mathbb{R}} f(y,t)\cdot |y|^{q-n}dy + C,
	\end{aligned}
	\end{equation}
	for some positive constant $C$ depending on $r_0$ and $\sup f$.
	Since $f\in \L^1$, if we write $f=e^{-\phi}$,  we have
	\begin{equation}
	\label{eq 2000}
		\liminf_{|x|\rightarrow \infty}\frac{\phi(x)}{|x|}>0.
	\end{equation}
	In particular, this implies the existence of $c_0>0$ and $M>0$ such that for all $|(y,t)|>M$, we have
	\begin{equation}
	\label{eq 1001}
		\phi(y,t)>c_0 |(y,t)|.
	\end{equation}
	The desired result follows from combining \eqref{eq 1000} and \eqref{eq 1001}.
	
	Let us now consider the case $q>n$. 	
	
	We first note that
	\begin{equation}
	\label{eq 1004}
		\int_{B(M^2)} |\nabla f(x)|\cdot |x|^{q-n}dx\leq M^{2(q-n)} \int_{B(M^2)}|\nabla f(x)|dx\leq M^{2(q-n)} \int_{\rn} |\nabla f(x)|dx<\infty,
	\end{equation}
	where the last inequality follows from the previously established case $q=n$. Therefore, we only need to show
	\begin{equation}
	\label{eq 2002}
		\int_{B(M^2)^c} |\nabla f(x)|\cdot |x|^{q-n}dx<\infty.
	\end{equation}
	Note that,
	\begin{equation}
	\label{eq 1005}
		\begin{aligned}
			&\int_{B(M^2)^c} \left|\frac{\partial f}{\partial x_n} (x)\right|\cdot |x|^{q-n}dx \\=&\sum_{k=M^2}^{\infty} \int_{B(k+1)\setminus B(k)}  \left|\frac{\partial f}{\partial x_n} (x)\right|\cdot |x|^{q-n}dx\\
			\leq & \sum_{k=M^2}^{\infty}(k+1)^{q-n} \int_{B(k+1)\setminus B(k)} \left|\frac{\partial f}{\partial x_n} (x)\right|dx\\
			\leq& \sum_{k=M^2}^{\infty}(k+1)^{q-n}\left(\int_{[B(k+1)\setminus B(k-1)]\cap \mathbb{R}^{n-1}}\int_{-\infty}^{\infty}+\int_{B(k-1)\cap \mathbb{R}^{n-1}}\int_{|t|\geq \sqrt{k}}\right) \left|\frac{\partial f}{\partial t}(y,t)\right|dtdy\\
					\end{aligned}
	\end{equation}
	
	For the first term, by \eqref{eq 1001}, we have
	\begin{equation}
	\label{eq 2003}
		\begin{aligned}
			&\sum_{k=M^2}^{\infty}(k+1)^{q-n}\int_{[B(k+1)\setminus B(k-1)]\cap \mathbb{R}^{n-1}}\int_{-\infty}^{\infty} \left|\frac{\partial f}{\partial t}(y,t)\right|dtdy\\
			\leq & 4\sum_{k=M^2}^{\infty}(k+1)^{q-n}\int_{[B(k+1)\setminus B(k-1)]\cap \mathbb{R}^{n-1}}\sup_{t\in \mathbb{R}}f(y,t)dy\\
			\leq & 4\sum_{k=M^2}^\infty (k+1)^{q-n}\int_{[B(k+1)\setminus B(k-1)]\cap \mathbb{R}^{n-1}} e^{-c_0 |y|}dy\\
			\leq & 4C\sum_{k=M^2}^\infty(k+1)^{q-n}e^{-c_0(k-1)}(k+1)^{n-1}<\infty.
		\end{aligned}
	\end{equation}
	
	For the second term, by using Lemma \ref{lemma log concave derivative control} again, we have
	\begin{equation}
		\int_{|t|\geq \sqrt{k}}\left|\frac{\partial f}{\partial t}(y,t)\right|dt\leq 4\sup_{|t|\geq \sqrt{k}} f(y,t),
	\end{equation}
	and consequently, by \eqref{eq 1001},
	\begin{equation}
	\label{eq 2004}
	\begin{aligned}
		&\sum_{k=M^2}^{\infty}(k+1)^{q-n}\int_{B(k-1)\cap \mathbb{R}^{n-1}}\int_{|t|\geq \sqrt{k}} \left|\frac{\partial f}{\partial t}(y,t)\right|dtdy\\
		\leq&  4\sum_{k=M^2}^{\infty}(k+1)^{q-n}\int_{B(k-1)\cap \mathbb{R}^{n-1}}\sup_{|t|\geq \sqrt{k}} f(y,t)dy\\
		\leq & 4\sum_{k=M^2}^{\infty}(k+1)^{q-n}\int_{B(k-1)\cap \mathbb{R}^{n-1}}e^{-c_0\sqrt{k}}dy\\
		=& 4C\sum_{k=M^2}^{\infty}(k+1)^{q-n} (k-1)^{n-1} e^{-c_0\sqrt{k}}<\infty.
	\end{aligned}	
	\end{equation}
	Equation \eqref{eq 2002} now follows from \eqref{eq 1005}, \eqref{eq 2003}, and \eqref{eq 2004}.
\end{proof}

\bibliographystyle{plain}

\end{document}